\documentclass[a4paper,11pt]{article}
\usepackage[utf8]{inputenc}

\usepackage{cmap}
\usepackage[T1]{fontenc}
\usepackage{diagbox}

\newcommand{\vertiii}[1]{{\left\vert\kern-0.25ex\left\vert\kern-0.25ex\left\vert #1 
    \right\vert\kern-0.25ex\right\vert\kern-0.25ex\right\vert}}

\usepackage{subcaption}

\usepackage{tocloft}

\usepackage{cite}

\usepackage[english]{babel}
\usepackage{amsfonts}
\usepackage{amssymb, amsthm}
\usepackage{xcolor,graphicx}
\usepackage{float}
\usepackage{marginnote}
\usepackage{amsmath}
\usepackage{a4wide}
\usepackage[affil-it]{authblk}

\usepackage{epstopdf}
\usepackage[margin=0.8in]{geometry}

\usepackage{titlesec}
\titlelabel{\thetitle.\quad}

\usepackage[labelsep=period]{caption}

\usepackage{enumitem}

\usepackage{crossreftools,ulem}
\makeatletter
\newcommand{\optionaldesc}[2]{%
	\phantomsection
	#1\protected@edef\@currentlabel{#1}\label{#2}%
}
\makeatother

\numberwithin{equation}{section}

\let\OLDthebibliography\thebibliography
\renewcommand\thebibliography[1]{
	\OLDthebibliography{#1}
	\setlength{\parskip}{1pt}
	\setlength{\itemsep}{1pt plus 0.3ex}
}

\usepackage[colorlinks=true,linktocpage,pdfpagelabels,
bookmarksnumbered,bookmarksopen]{hyperref}
\definecolor{ForestGreen}{rgb}{0.1,0.6,0.05}
\definecolor{EgyptBlue}{rgb}{0.063,0.1,0.6}
\hypersetup{
	colorlinks=true,
	linkcolor=EgyptBlue,         
	citecolor=ForestGreen,
	urlcolor=olive
}

\usepackage[hyperpageref]{backref}

\allowdisplaybreaks
\usepackage{mathtools}
\mathtoolsset{showonlyrefs}

\usepackage{orcidlink}

\usepackage{accents}

\newcommand{\al} {\alpha}
\newcommand{\pa} {\partial}
\newcommand{\be} {\beta}

\newcommand{\Om} {\Omega}
\newcommand{\la} {\lambda}

\newcommand{\tht} {\theta}

\newcommand{\si} {\sigma}

\newcommand{\noi} {\noindent}
\newcommand{\R}{{\mathbb R}}

\newcommand{\dx}{{\,\rm d}x}

\newcommand{\dsi}{{\,\rm d}\si}

\newcommand{\Omo}{\Om_{\rm{out}}}
\newcommand{\Omi}{\Om_{\rm{in}}}

\newcommand{\ho}{h_{\mathrm{out}}}
\newcommand{\hi}{h_{\mathrm{in}}}

\newtheorem{theorem}{Theorem}[section]
\newtheorem{lemma}[theorem]{Lemma}
\newtheorem{proposition}[theorem]{Proposition}

\theoremstyle{definition}
\newtheorem{definition}[theorem]{Definition}
\newtheorem{remark}[theorem]{Remark}

\title{\textsc{On the Hersch--Weinberger inequality \\in higher dimensions}
}

\author{T.~V.~Anoop, Vladimir~Bobkov, Mrityunjoy~Ghosh, Olga~Pochinka}
\date{}

\AtEndDocument{%
	\par
	\medskip
	\bigskip
	\begin{tabular}{@{}l@{}}%
(T. V.~Anoop)\\[0.2em]
		\textsc{Department of Mathematics, Indian Institute of Technology Madras}\\
		\textsc{Chennai 36, India}\\[0.3em]
		\orcidlink{0000-0002-2470-9140} 0000-0002-2470-9140\\[0.3em]
		\textit{E-mail address}: \texttt{anoop@iitm.ac.in}\\
        [1.5em]
        
		(V.~Bobkov)\\[0.2em]
		\textsc{Institute of Mathematics, Ufa Federal Research Centre, RAS}\\
		\textsc{Chernyshevsky str. 112, 450008 Ufa, Russia}\\[0.3em]
		\orcidlink{0000-0002-4425-0218} 0000-0002-4425-0218\\[0.3em]
		\textit{E-mail address}: \texttt{bobkov@matem.anrb.ru, bobkovve@gmail.com}\\[1.5em]
		
		(M.~Ghosh)\\[0.2em]
		\textsc{Tata Institute of Fundamental Research,
Centre for Applicable Mathematics}\\
		\textsc{Sharadanagar, Bengaluru 560065, India}\\[0.3em]
		\orcidlink{0000-0003-0415-2821} 0000-0003-0415-2821\\[0.3em]
		\textit{E-mail address}: \texttt{ghoshmrityunjoy22@gmail.com}\\
        [1.5em]
        
        (O.~Pochinka)\\[0.2em]
		\textsc{International Laboratory of Dynamical Systems and Applications, HSE University}\\
		\textsc{Bolshaya Pecherckaya str.\ 25/12, 603155 Nizhny Novgorod, Russia}\\[0.3em]
		\orcidlink{0000-0002-6587-5305} 0000-0002-6587-5305\\[0.3em]
		\textit{E-mail address}: \texttt{olga-pochinka@yandex.ru}
\end{tabular}}

\begin{document}
\maketitle
	\vspace*{-5ex}
	\begin{abstract}

		We investigate a reverse Faber--Krahn type inequality for the Robin Laplacian in a bounded smooth domain $\Omega \subset \mathbb{R}^N$ whose boundary has two connected components.  
        We prove that a concentric spherical shell maximizes the first eigenvalue over a class of such domains under perimeter and volume constraints, and under an additional convexity assumption when $N \geq 3$. This result generalizes to a wider class, and extends to higher dimensions, the inequality of \textsc{Hersch} \cite{hersh}, whose approach was substantially based on a construction of the so-called \textit{effectless cut} by \textsc{Weinberger} \cite{wein}, 
        so that we call it the Hersch--Weinberger inequality. Our method is based on the analysis of the gradient flow of the first eigenfunction and several approximation procedures, without relying on the effectless cut itself. The effectless cut being a complicated object related to the attractor of the gradient flow, we describe its most fundamental topological properties. In particular, we show that it does not necessarily have to be a hypersurface.
				
		
	\end{abstract}
	
	\renewcommand{\cftdot}{.}
	\begin{quote}	
		\tableofcontents	
		\addtocontents{toc}{\vspace*{-2ex}}
	\end{quote}

	\section{Introduction}\label{sec:intro}
	
    Let $\Omega \subset \mathbb{R}^N$ be a bounded domain of class $C^{2,\tht},$ for some $\tht \in (0,1)$ and $N \geq 2$, that can be written as
	\begin{equation}\label{def:omega}
	    \Omega = \Omo \setminus \overline{\Omi},
	\end{equation} 
where the domain $\Omi$ is compactly contained in the domain $\Omo$, and the boundaries $\pa \Omi$, $\pa \Omo$ are connected sets. 
Consider the eigenvalue problem 
\begin{align}\tag{$\mathcal{RR}$}\label{eq:D}
\left\{\begin{aligned}
    -\Delta u &= \lambda u \quad\text{in} \quad\Omega,\\
		\frac{\pa u}{\pa \nu}+ \hi u&=0 \quad\;\;\text{on}\quad  \pa \Omi,\\
		\frac{\pa u}{\pa \nu}+ \ho u &=0 \quad\;\; \text{on}\quad  \pa \Omo,
	\end{aligned}
	 \right.
\end{align}
where $\nu$ is the unit outward normal vector to $\pa \Om$ and $\hi,\ho \in [0, +\infty]$.
The case $\hi=+\infty$ or $\ho=+\infty$ corresponds to the zero Dirichlet boundary conditions on $\pa \Omi$ or $\pa \Omo$, respectively. 

The spectral theorem for self-adjoint compact operators ensures that the spectrum of \eqref{eq:D} is discrete and unbounded. Moreover, 
the first eigenvalue $\lambda_1(\Om)$ admits the following variational characterization:
\begin{equation}\label{eq:lambda1}
	\lambda_1(\Om)
=
\inf_{v \in \widetilde{H}^1(\Omega)\setminus\{0\}}
\frac{\int_{\Om}|\nabla v|^2 \dx+\hi \int_{\pa \Omi} v^2 \dsi+\ho \int_{\pa \Omo} v^2 \dsi}{\int_{\Om} v^2 \dx},
	\end{equation}
	where 
     \begin{equation}\label{eq:H1}
 \widetilde{H}^1(\Omega):=\left\{v\in H^1(\Omega): v=0 \text{ on } \partial \Omi \text{ if } \hi=+\infty, \text{ and } v=0 \text{ on } \partial \Omo \text{ if } \ho=+\infty \right\},
 \end{equation}
 and we set $\hi \int_{\pa \Omi} v^2 \dsi := 0$ if $\hi=+\infty$, and $\ho \int_{\pa \Omo} v^2 \dsi := 0$ if $\ho=+\infty$.  
It is clear that $\lambda_1(\Om)>0$ provided at least one Robin parameter is nonzero. 
The infimum is attained, the corresponding minimizer (the first eigenfunction) is unique up to scaling, has a fixed strict sign in $\Omega$, is real-analytic in $\Omega$, and belongs to $C^{2,\tht}(\overline{\Omega})$, see, e.g., \cite[Theorem~6.31]{gilbargtrudinger} for the last fact.

Throughout the text, for convenience of exposition, we adopt the following convention to denote the first eigenvalue. 
We write 
\begin{equation}\label{eq:lllll}
\lambda_1^{\mathcal{RR}}(\Om),~ 
\lambda_1^{\mathcal{NR}}(\Om),~ 
\lambda_1^{\mathcal{RN}}(\Om),~
\end{equation}
where the superscript $\mathcal{RR}$ corresponds to the Robin-Robin case $\hi,\ho \in (0,+\infty]$, $\mathcal{NR}$ to the Neumann-Robin case $\hi=0$, $\ho \in (0,+\infty]$, and $\mathcal{RN}$ to the Robin-Neumann case $\hi \in (0,+\infty]$, $\ho =0$. 
In the historical overview given below, we also use the notation 
$$
\lambda_1^{\mathcal{DN}}(\Om),~
\lambda_1^{\mathcal{ND}}(\Om),~
\lambda_1^{\mathcal{DD}}(\Om),
$$
with the Dirichlet boundary conditions instead of the Robin ones on the respective parts of the boundary. 

By $|\cdot|$ we denote either the Lebesgue measure or the $(N-1)$-dimensional Hausdorff measure of a set, the precise meaning being clear from the context. 
Also, $B_r$ stands for the open ball of radius $r>0$ centered at the origin, unless otherwise explicitly specified. 

\medskip
 In this article, we are interested in establishing an isoperimetric-type inequality of the form
\begin{equation}\label{eq:reverseRR}
		\lambda_1^{\mathcal{RR}}(\Omo \setminus \overline{\Omi})
		\leq
		\lambda_1^{\mathcal{RR}}(B_\beta \setminus \overline{B_\alpha}),
	\end{equation}
    where the choice of the concentric balls $B_\beta$ and $B_\alpha$ is determined by the dimension $N$, the Robin parameters $\hi, \ho \in (0,+\infty]$, and some natural geometric constraints. 
    This inequality will be referred to as the \textit{Hersch--Weinberger inequality}. 

In the planar case $N=2$, for the Neumann-Dirichlet boundary conditions ($\hi=0$ and $ \ho=+\infty$), \textsc{Payne \& Weinberger} \cite{PayneWein1961} proved the inequality
	\begin{equation}\label{eq:reverseND}
		\lambda_1^{\mathcal{ND}}(\Omo \setminus \overline{\Omi})
		\leq
		\lambda_1^{\mathcal{ND}}(B_\beta \setminus \overline{B_\alpha}),
	\end{equation}
	where $B_\beta$ and $B_\alpha$ satisfy 
  \begin{equation}\label{eq:reverseFK2d-assumption-ND}
	|\partial B_\beta| = |\partial \Omo|,
	\quad 
	|\Omo \setminus \overline{\Omi}| 
    = |B_\beta \setminus \overline{B_\alpha}|,
	\end{equation}	
    and they also covered the Neumann-Robin case. 
    Later, \textsc{Hersch} \cite{hersh} provided a different argument for \eqref{eq:reverseND} and proved a similar inequality for the Dirichlet-Neumann boundary conditions ($\hi=+\infty$ and $ \ho=0$), namely, 
    \begin{equation}\label{eq:reverseDN}
		\lambda_1^{\mathcal{DN}}(\Omo \setminus \overline{\Omi})
		\leq
		\lambda_1^{\mathcal{DN}}(B_\beta \setminus \overline{B_\alpha}),
	\end{equation}
    for $B_\beta$ and $B_\alpha$ satisfying 
    	\begin{equation}\label{eq:reverseFK2d-assumption-DN}
    |\partial B_\alpha| = |\partial \Omi|, 
	\quad 
	|\Omo \setminus \overline{\Omi}| 
    = |B_\beta \setminus \overline{B_\alpha}|.
	\end{equation}	
    In order to prove these inequalities, both \textsc{Payne \& Weinberger} and \textsc{Hersch} combined the interior parallel method with Nagy's type inequalities (see \cite{Nagy1959}). 

In the planar case, for the Dirichlet-Dirichlet boundary conditions ($\hi=\ho=+\infty$), to derive an isoperimetric inequality of the stated form \eqref{eq:reverseRR}, one might expect  $B_\beta$ and $B_\alpha$ to satisfy both \eqref{eq:reverseFK2d-assumption-ND} and \eqref{eq:reverseFK2d-assumption-DN}, i.e.,
\begin{equation}\label{eq:reverseFK2d-assumption}
    |\partial B_\alpha| = |\partial \Omi|, 
	\quad
	|\partial B_\beta| = |\partial \Omo|,
	\quad 
	|\Omo \setminus \overline{\Omi}| 
    = |B_\beta \setminus \overline{B_\alpha}|.
	\end{equation}
Notice that since  $B_\beta$ and $B_\alpha$ necessarily satisfy the identity 
$$
|\partial B_\beta|^2_1-|\partial B_\alpha|^2_1 = 4\pi |B_\beta \setminus \overline{B_\alpha}|,
$$ 
the domain $\Omega=\Omo \setminus \overline{\Omi}$ satisfies \eqref{eq:reverseFK2d-assumption} if and only if the following isoperimetric constraint holds:
	\begin{equation}\label{eq:reverseFK2d-assumption0}
	|\partial \Omo|^2-|\partial \Omi|^2
	= 
	4\pi |\Omo \setminus \overline{\Omi}|.
	\end{equation}
    In fact, for such a domain $\Omega=\Omo \setminus \overline{\Omi}$, \textsc{Hersch} \cite{hersh} established the inequality
    \begin{equation}\label{eq:reverseDD}
		\lambda_1^{\mathcal{DD}}(\Omo \setminus \overline{\Omi})
		\leq
		\lambda_1^{\mathcal{DD}}(B_\beta \setminus \overline{B_\alpha}).
	\end{equation}
    The concept of an ``\textit{effectless cut}'', introduced by \textsc{Weinberger} \cite{wein}, forms the basis of \textsc{Hersch}'s proof. 
    For this reason, we call \eqref{eq:reverseDD} and, more generally, \eqref{eq:reverseRR}, \textit{the Hersch--Weinberger inequality}. 
    
    In two dimensions, the effectless cut is a closed curve $\widetilde{\gamma}$ within \(\Omega\) that surrounds \(\Omi\) and divides \(\Omega\) into two doubly connected subdomains, \(\Omega_1\) (inner) and \(\Omega_2\) (outer), without affecting the first eigenvalue of the Laplacian (with mixed boundary conditions) on either \(\Omega_1\) or \(\Omega_2\), see Figure~\ref{fig:cut}. 
    More precisely, we have the identity
$$
\lambda_1^{\mathcal{DN}}(\Omega_1) = \lambda_1^{\mathcal{ND}}(\Omega_2) = \lambda_1^{\mathcal{DD}}(\Omega),
$$
which explains the term ``effectless cut''. 
\textsc{Hersch} then combined the inequalities \eqref{eq:reverseND} for $\Omega_2$ and \eqref{eq:reverseDN} for $\Omega_1$  to conclude the desired inequality \eqref{eq:reverseDD}. However, the rigorous implementation of this approach requires the effectless cut $\widetilde{\gamma}$ to be sufficiently regular in order for the eigenvalue problem on the subdomains to be well-posed. 
This subtle yet important issue was not discussed in \cite{hersh,wein}, see also \cite[Theorem~3.5.3]{henrot2006extremum}.
In the recent work \cite{ABG2}, we revised \cite{hersh,wein}, investigated the regularity of the effectless cut, introduced its regularized version,  
and generalized the inequality \eqref{eq:reverseDD} to the Robin-Robin case in which both Robin parameters can be positive or negative.

\begin{figure}[!ht]
    \centering
\includegraphics[width=0.4\linewidth]{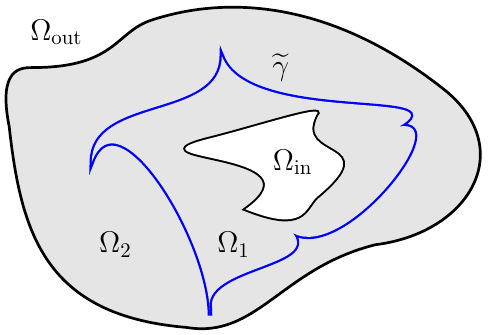}
\captionsetup{width=0.75\textwidth}
    \caption{The subdomains $\Om_1$ and $\Om_2$. Here, $\widetilde{\gamma}$ represents the \textit{effectless cut}.}
    \label{fig:cut}
\end{figure} 
 
In higher dimensions $N \geq 3$, in the Neumann-Dirichlet case, the reverse Faber--Krahn inequality \eqref{eq:reverseND} was established in \cite{AnoopAshok2020}, assuming that \(\Omo\) is a ball and $\Omo \setminus \overline{\Omi}$ satisfies an isoperimetric constraint analogous to \eqref{eq:reverseFK2d-assumption-ND}, namely,
	\begin{equation}\label{eq:reverseFK-assumption-ND}
    |\partial B_\beta| = |\partial \Omo|, 
	\quad 
	|\Omo \setminus \overline{\Omi}| 
    = |B_\beta \setminus \overline{B_\alpha}|.
	\end{equation}
 This result was further extended to the Neumann-Robin case in \cite[Theorem 3.1]{PaoliPiscitelli}, with a weaker assumption that \(\Omo\) is convex together with \eqref{eq:reverseFK-assumption-ND}. 
 In \cite[Theorem~1.2]{AnoopMrityunjoy}, a different proof of this result was provided for the Neumann-Dirichlet case, using a higher-dimensional analogue of Nagy's inequality. 
 For the Dirichlet-Neumann case, \eqref{eq:reverseDN} was established in \cite{AnoopAshok2020} again under the assumption that \(\Omi\) is a ball together with the following isoperimetric constraint analogous to \eqref{eq:reverseFK2d-assumption-DN}:
 	\begin{equation}\label{eq:reverseFK-assumption-DN0}
    |\partial B_\alpha| = |\partial \Omi|, 
	\quad 
	|\Omo \setminus \overline{\Omi}| 
    = |B_\beta \setminus \overline{B_\alpha}|.
	\end{equation}	
 However, to the best of our knowledge, the validity of \eqref{eq:reverseDN} remains unknown 
 for more general \(\Omi\) satisfying the constraint \eqref{eq:reverseFK-assumption-DN0}.
 
 In \cite{DellaPiscitelli}, the authors assumed that $\Omi$ is convex and proposed another isoperimetric constraint in place of \eqref{eq:reverseFK2d-assumption-DN}:
 \begin{equation}\label{eq:reverseFK-assumption-DN}
    W_{N-1}(B_\al)=W_{N-1}(\Omi), 
	\quad 
	|\Omo \setminus \overline{\Omi}|
    = |B_\beta \setminus \overline{B_\alpha}|,
	\end{equation}
where $W_{N-1}(E)$ is the $(N-1)$-th quermassintegral of a convex domain $E$, see, e.g., 
\cite[Chapter~5]{Schneider} or Section~\ref{sec:K} below for more details. 
    Notice that, for $N=2$, $W_1(\Omega) = |\partial \Omega|/2$, and hence both isoperimetric constraints \eqref{eq:reverseFK2d-assumption-DN} and \eqref{eq:reverseFK-assumption-DN} are equivalent. 
    On the other hand, for $N\ge 3$, the constraints \eqref{eq:reverseFK-assumption-DN0}  and \eqref{eq:reverseFK-assumption-DN} yield, respectively, two concentric spherical shells $\mathcal{A}_1$ and $\mathcal{A}_2$ such that $\lambda_1^{\mathcal{DN}}(\mathcal{A}_1)<\lambda_1^{\mathcal{DN}}(\mathcal{A}_2)$; see Remark~\ref{rem:perimeter_constraint} below for more details. 
    In \cite[Theorem~1.1]{DellaPiscitelli}, the authors established \eqref{eq:reverseRR} for the Robin-Neumann boundary conditions. 
    A different proof for the Dirichlet-Neumann case was proposed in \cite[Theorem~1.6]{AnoopMrityunjoy}. 
    
    By contrast, higher-dimensional versions of the Hersch--Weinberger inequality \eqref{eq:reverseRR} were known only for rather narrow classes of domains. 
    For instance, see \cite{ABS2018,HKK2001,Kesavan2003} for the class of eccentric spherical shells.
    A higher-dimensional generalization of \eqref{eq:reverseRR} in the Dirichlet-Robin case was claimed in \cite[Theorem~1.2]{GP}, assuming that $\Omo$ and $\Omi$ are convex and that there exist $\alpha, \beta >0 $ such that
    \begin{equation}\label{eq:reverseFK-assumption}
	W_{N-1}(B_\alpha) = W_{N-1}(\Omi),
	\quad
    |\partial B_\beta| = |\partial \Omo|,
	\quad 
	|\Omo \setminus \overline{\Omi}| = |B_\beta \setminus \overline{B_\alpha}|.
	\end{equation} 
    Observe that \eqref{eq:reverseFK-assumption} is equivalent to \eqref{eq:reverseFK2d-assumption} for $N=2$. 
	The approach of \cite{GP} is based on the web-function method, which constructs a suitable test function defined by means of parallel sets (see, for instance, \cite{hersh, AnoopMrityunjoy, AnoopAshok2020, PayneWein1961, PaoliPiscitelli}). 
The authors specifically aim to construct an appropriate test function using the distance functions to both the inner and outer boundaries. However, the web function constructed in the proof of \cite[Theorem~1.2]{GP} may exhibit a jump discontinuity and therefore does not generally belong to the space \( \widetilde{H}^1(\Omega) \). 
Recently, the inequality \eqref{eq:reverseRR} has been claimed in \cite{AmatoGavitone} for domains with axial symmetry and a convexity assumption, satisfying the constraints \eqref{eq:reverseFK-assumption}. 
The arguments of \cite{AmatoGavitone} pursue a generalization of the effectless cut approach from \cite{hersh,wein} and \cite{ABG2} to the axially symmetric higher-dimensional setting. 

In the present work, we establish the Hersch--Weinberger inequality \eqref{eq:reverseRR} for a more general class of domains in arbitrary dimension. 
For $N \geq 2$ and $\beta > \alpha>0$, we consider the class of domains $\mathcal{K}_{\alpha,\beta}^N$ defined as follows: 
    \begin{align}
    \mathcal{K}_{\alpha,\beta}^2
    = 
    \Big\{
    &\Omega = \Omo \setminus \overline{\Omi} \subset \mathbb{R}^2:~\overline{\Omi}\subset \Omo
    ~\text{and}
    \\
    &|\partial B_\alpha| = |\partial \Omi|,~~ 
	|\partial B_\beta| = |\partial \Omo|,~~
	|\Omega| \geq  |B_\beta \setminus \overline{B_\alpha}|
    \Big\},
    \end{align}
    and, for $N \geq 3$,
    \begin{align}
    \mathcal{K}_{\alpha,\beta}^N
    = 
    \Big\{
    &\Omega = \Omo \setminus \overline{\Omi} \subset \mathbb{R}^N:~\overline{\Omi}\subset \Omo,
    ~\text{$\Omi,\Omo$ are convex, and}
    \\
    &W_{N-1}(B_\alpha) = W_{N-1}(\Omi),~~ 
	|\partial B_\beta| = |\partial \Omo|,~~
	|\Omega| \geq  |B_\beta \setminus \overline{B_\alpha}|
    \Big\}.
    \end{align}
    Recall that we assume, by default, that any $\Omega$ is $C^{2,\tht}$-regular. 
    A few points are worth mentioning:
    \begin{enumerate}[label={\rm(\roman*)}]
        \item For each $\Omega,$ there is at most one choice of $\beta>\alpha>0$ such that $\Omega \in \mathcal{K}_{\alpha,\beta}^N$. 
        \item If $\Omega \in \mathcal{K}_{\alpha,\beta}^N$, then $\Omega$ is homeomorphic to the spherical shell $B_\beta \setminus \overline{B_\alpha}$. 
    \end{enumerate}
    Clearly, $B_\beta \setminus \overline{B_\alpha} \in \mathcal{K}_{\alpha,\beta}^N$. 
    Further examples of domains in $\mathcal{K}_{\alpha,\beta}^N$ and additional information on this class will be discussed in Section~\ref{sec:K}. 

    \smallskip
    We now state our main result. 
    \begin{theorem}\label{thm:FK}
	Let $N \geq 2$ and $\hi,\ho \in (0,+\infty]$.
	Let $\beta>\alpha>0$.
	   If $\Omega \in \mathcal{K}_{\alpha,\beta}^N$, then 
       \begin{equation}\label{eq:reverseFK}
		\lambda_1^{\mathcal{RR}}(\Omo \setminus \overline{\Omi})
		\leq
		\lambda_1^{\mathcal{RR}}(B_\beta \setminus \overline{B_\alpha}).
	\end{equation}
    \end{theorem}

    In Section~\ref{sec:counterexample}, we show that \eqref{eq:reverseFK} can be reversed for domains $\Omega=\Omo \setminus \overline{\Omi}$ satisfying
    \begin{equation}\label{eq:counterexampl1}
	W_{N-1}(B_\alpha) = W_{N-1}(\Omi),
	\quad
    |\partial B_\beta| = |\partial \Omo|,
	\quad 
	|\Omo \setminus \overline{\Omi}| < |B_\beta \setminus \overline{B_\alpha}|,
    \end{equation}
    and such that $\Omi$, $\Omo$ are convex. 
    This indicates that Theorem~\ref{thm:FK} is, in this sense, optimal.
    
      \medskip
    Our method of the proof of Theorem~\ref{thm:FK} is different from that in \cite{hersh,wein}, although it has a similar underlying idea of analyzing the gradient flow associated with the first eigenfunction.  
    Recall that the key ingredient in \textsc{Hersch}'s work \cite{hersh} is the \textit{effectless cut} constructed by \textsc{Weinberger} \cite{wein} via such a gradient flow, specifically in the planar case \(N=2\). 
     The approach of \cite{hersh,wein} was recently revised in \cite{ABG2}. 
    However, the arguments of \cite{ABG2,hersh,wein} are difficult to translate directly to the higher-dimensional case $N \geq 3$. 
    This is mainly due to the fact that a precise analogue of the effectless cut in higher dimensions can be much more ``wild'' than in the planar case. 
    Namely, it might not be a hypersurface even when the first eigenfunction is a Morse--Smale function, as we show in Section~\ref{sec:surface}.    
    Consequently, the analysis of the well-posedness of the eigenvalue problem \eqref{eq:D} on the corresponding subdomains is a highly nontrivial (yet very interesting) issue. 
   
    In the present work, we overcome these regularity difficulties using two approximation procedures. 
    We approximate the first eigenfunction $u$ of \eqref{eq:D} by the first eigenfunctions of the problem with a potential (see Proposition~\ref{prop:approx}), such that the approximate eigenfunctions are Morse functions.  
    Then, instead of investigating the effectless cut itself, we analyze the behavior of the boundaries $\pa\Omi$, $\pa\Omo$ under the gradient flow of these Morse eigenfunctions (see Section~\ref{sec:approx:cut}) and study the behavior of Robin-Neumann and Neumann-Robin eigenvalues in the corresponding subdomains. 
    Eventually, passing to the limit along both approximations, we derive the Hersch--Weinberger inequality. 
    This method provides a robust framework that avoids the consideration of the regularity of the effectless cut. 
    Our approach works in any dimension $N \geq 2$, thus also providing a new proof of the Hersch--Weinberger inequality \eqref{eq:reverseFK} in the planar case, cf.~\cite[Theorem~1.2]{ABG2} and  \cite{hersh}. 

\medskip
The article is organized as follows. 
In Section~\ref{sec:K}, we recall the notion of quermassintegrals and discuss instances of domains belonging to the admissible class $\mathcal{K}_{\alpha,\beta}^N$ for Theorem~\ref{thm:FK}. 
An example of a domain satisfying the constraints \eqref{eq:counterexampl1} but violating the Hersch--Weinberger inequality is given in Section~\ref{sec:counterexample}. 
Section~\ref{sec:proof_FK} is devoted to the proof of Theorem~\ref{thm:FK}. 
In Section~\ref{sec:surface}, we propose a definition of the effectless cut in higher dimensions and discuss its most fundamental topological properties. In particular, we show that the effectless cut does not have to be a hypersurface.
Finally, we provide a few concluding remarks in Section~\ref{sec:final_remarks}.

\section{On the admissible class of domains
}\label{sec:K}

We start by recalling the notion of quermassintegrals of a bounded convex domain $E \subset \mathbb{R}^N$, see \cite{HugWein,Schneider}. 
While this notion is defined classically for \textit{compact} convex sets, it works equally well for the convex domain $E$ since $\partial E = \partial \overline{E}$. 
For $\delta>0$, denote by $E_\delta$ the $\delta$-neighborhood of $E$, also known as the \textit{outer parallel body} of $E$:
\begin{equation}\label{eq:Kdelta}
E_\delta
=
E + \delta B_1
=
\{
x \in \mathbb{R}^N:~ \mathrm{dist}(x,E) < \delta
\}.
\end{equation}
Here, the sum of the sets is the Minkowski sum.
The \textit{Steiner formula} (see, e.g., \cite[Section~4.1, Eq.~(4.1)]{Schneider}) asserts that the measure of $E_\delta$ can be expressed as a polynomial of degree $N$ in the variable $\delta$:
\begin{equation}\label{eq:steinerRN}
|E_\delta|
=
\sum_{i=0}^{N} \binom{N}{i} W_i(E) \,\delta^i,
\end{equation}
where the coefficients $W_i(E)$ are called the \textit{quermassintegrals} of $E$. 
It is known that
\begin{equation}\label{eq:WWW}
W_0(E) = |E|, 
\quad 
W_1(E) = \frac{1}{N} \, |\partial E|,
\quad
W_N(E) = |B_1|.
\end{equation}
In particular, for $N=2$, \eqref{eq:steinerRN} gives 
\begin{equation}\label{eq:steiner2D}
|E_\delta| = |E| + |\partial E| \delta + \pi \delta^2. 
\end{equation}
There are several ways to characterize $W_i(E)$ with $i \in \{2,\dots,N-1\}$, but they are generally not as simple as those in \eqref{eq:WWW}, see, e.g., \cite[Section~5.3.1]{Schneider}. 
Let us mention that if $E$ is a ball $B_r$ in $\mathbb{R}^N$ of radius $r>0$, then (see \cite[Eq.~(13.46)]{santalo}) 
\begin{equation}\label{eq:quer-ball}
W_{N-1}(B_r) = |B|_N \, r.
\end{equation}

We also recall the following \textit{Alexandrov--Fenchel inequalities} for quermassintegrals (see, e.g., \cite[Section~7.4, (7.67)]{Schneider}):
\begin{equation}
 	\label{Alexandrov_Fenchel}
 	\left(\frac{W_j(E)}{|B_1|}\right)^{\frac{1}{N-j}}\geq \left(\frac{W_i(E)}{|B_1|}\right)^{\frac{1}{N-i}} \quad 
    \text{for}~0\leq i<j<N,
 \end{equation}
where equality holds for some $i$ and $j$ if and only if $E$ is a ball. 
In particular, for $i=0$ and $j=1$, we get from \eqref{eq:WWW} and \eqref{Alexandrov_Fenchel} the classical isoperimetric inequality 
\begin{equation}\label{eq:isoperimetric-classic}
    |\partial E|\ge N|B_1|^\frac{1}{N} |E|^{\frac{N-1}{N}}.
\end{equation}

\subsection{Members of \texorpdfstring{$\mathcal{K}_{\alpha,\beta}^N$}{K}}
\label{section:members}
Let us discuss examples of domains belonging to $\mathcal{K}_{\alpha,\beta}^N$. 
It is obvious that spherical shells bounded by two balls of radii $\alpha < \beta$, not necessarily concentric, belong to $\mathcal{K}_{\alpha,\beta}^N$.
We now turn to less trivial examples. 

\noi \underline{\bf $\Omo$ is a ball:} Fix $\beta>0$ and consider any domain of the form $B_\beta \setminus \overline{\Omi}$, where 
    ${\Omega}_{\mathrm{in}}$ is convex with $\overline{\Omi} \subset B_\beta$ (see Figure \ref{fig:example}-(a)). Then there exists $\alpha \in (0,\beta)$ (see Remark~\ref{rem:convex} below) such that  $W_{N-1}(B_\alpha) = W_{N-1}(\Omi)$.
   Now, the Alexandrov--Fenchel inequality \eqref{Alexandrov_Fenchel} for $i=0$ and $j=N-1$, together with equality for balls, yields
    $$
    |B_\alpha| \geq |\Omi|,
    $$
    see, e.g., \cite[Proposition 3.4]{AnoopMrityunjoy}. 
    Therefore, we obtain
    \begin{equation}\label{eq:BOmineq1}
    |B_\beta \setminus \overline{\Omi}| 
    =
    |B_\beta| -|\Omi|
    \geq 
    |B_\beta| -|B_\alpha|
    =  
    |B_\beta \setminus \overline{B_\alpha}|,
    \end{equation}
    which implies that $B_\beta \setminus \overline{\Omi} \in \mathcal{K}_{\alpha,\beta}^N$. 
    Note that equality holds in \eqref{eq:BOmineq1} if and only if $\Omi = B_\alpha$, where $B_\alpha$ and $B_\beta$ need not be concentric. 
    Thus, for every convex domain $\Omi$ satisfying $\overline{\Omi} \subset B_\beta$, the domain $B_\beta \setminus \overline{\Omi}$ belongs to $\mathcal{K}_{\alpha,\beta}^N$.

\begin{figure}[t]
\centering
  \begin{subfigure}[c]{0.3\textwidth}
    \includegraphics[width=\linewidth]{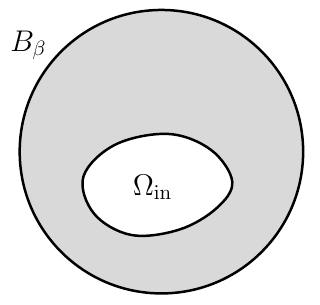}
    \caption{$B_\beta\setminus \overline{\Omi}$} \label{fig:2a}
  \end{subfigure}%
  \hspace{0.22\textwidth}
  \begin{subfigure}[c]{0.37\textwidth}
    \includegraphics[width=\linewidth]{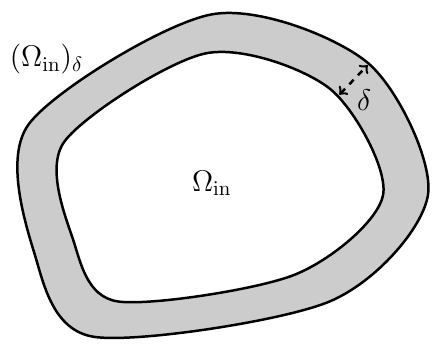}
    \caption{$(\Omi)_\delta\setminus \overline{\Omi}$} \label{fig:2b}
  \end{subfigure}
\caption{Examples of domains in $\mathcal{K}_{\alpha,\beta}^N$.} 
\label{fig:example}
\end{figure}

\noi \underline{\bf $\Omi$ is a ball:}
Let us take some $\alpha>0$ and consider a domain of the form $\Omo \setminus \overline{B_\alpha}$, where $\Omo$ is not necessarily convex and satisfies $\overline{B_\alpha} \subset \Omo$ and $|\partial B_\beta| = |\partial \Omo|$ for some $\beta > \alpha$. 
The isoperimetric inequality \eqref{eq:isoperimetric-classic} gives 
$|B_\beta| \geq |\Omo|$, 
where equality holds if and only if $\Omo = B_\beta$, 
which implies that
    \begin{equation}\label{eq:badK}
    |\Omo \setminus \overline{B_\alpha}| 
    =
    |\Omo| -|B_\alpha|
    \leq 
    |B_\beta| -|B_\alpha|
    =  
    |B_\beta \setminus \overline{B_\alpha}|.
    \end{equation} 
Thus, a domain of the form $\Omo \setminus \overline{B_\alpha}$ belongs to $\mathcal{K}_{\alpha,\beta}^N$ if and only if $\Omo = B_\beta$, where $B_\alpha$ and $B_\beta$ can be eccentric. 
\begin{remark}\label{rem:convex}
For $N=2$, one can easily find $\Omi$ and $\Omo$ such that $\overline{\Omi}\subset \Omo$
and $|\partial \Omi| \geq |\partial \Omo|$. 
In this case, we have $\alpha \geq \beta$, which means that such $\Omega$ does not belong to $\mathcal{K}_{\alpha,\beta}^2$. 
On the other hand, for general $N \geq 2$, 
if $\Omi$ and $\Omo$ are \textit{convex} and $\overline{\Omi}\subset\Omo$, then from \cite[Definition~3.5 and Theorem~3.9 (c), (g)]{HugWein} we have $W_{N-1}(\Omi)<W_{N-1}(\Omo)$. In particular, if $\overline{\Omi} \subset B_\beta$, then by \eqref{eq:quer-ball} there exists $\alpha \in (0,\beta)$ such that $W_{N-1}(B_\alpha) = W_{N-1}(\Omi)$.
\end{remark}

\begin{remark}
Clearly, if $\Omega$ belongs to $\mathcal{K}_{\alpha,\beta}^N$, then any translation and orthogonal transformation of $\Omo$ or $\Omi$ 
preserving the positive distance between connected components of the boundary keep the modified domain in the class $\mathcal{K}_{\alpha,\beta}^N$. 
\end{remark}

In the planar case $N=2$, other examples of domains in ${\mathcal{K}}_{\alpha,\beta}^2$ can be easily constructed (see, for instance, \cite[Section~3]{DNT2022}). 
Indeed, taking any bounded convex domain $\Omi$, fixing some $\delta>0$, and considering the $\delta$-neighborhood $(\Omi)_\delta$ of $\Omi$,
the Steiner formula \eqref{eq:steiner2D} and its derivative 
give the following expressions for the measure and perimeter of $(\Omi)_\delta$: 
$$
|(\Omi)_\delta| = |\Omi| + |\partial \Omi| \delta + \pi \delta^2, 
\quad 
|\partial (\Omi)_\delta| = |\partial \Omi| + 2 \pi \delta.
$$
It is not hard to deduce from these identities that $(\Omi)_\delta \setminus \overline{\Omi} \in {\mathcal{K}}_{\alpha,\beta}^2$ (see Figure \ref{fig:example}-(b)), where $\alpha>0$ is such that $|\partial B_\alpha| = |\partial \Omi|$ and $\beta = \alpha+ \delta$. 

In the higher-dimensional case \(N \geq 3\), a similar argument based on the Steiner formula \eqref{eq:steinerRN} does not seem to directly yield nontrivial examples of domains in \({\mathcal{K}}_{\alpha,\beta}^N\). 
This is primarily due to the challenge of computing \(W_{N-1}\). 
However, a way of constructing such domains is indicated in \cite{AmatoGavitone}. 
Namely, let us take any convex domains $\Omi$, $\Omo$, which are not balls, and take two balls $B_r$, $B_R$ such that
$$
\overline{B_r} \subset \Omi, 
\quad 
\overline{\Omi} \subset \Omo, 
\quad 
\overline{\Omo} \subset B_R. 
$$
Additionally assume that, for every $a, b \in [0,1]$, the domains
$$
\Gamma_{\text{out}}(a)
:=
a B_R + (1-a) \Omo,
\quad 
\Gamma_{\text{in}}(b)
:=
b B_r + (1-b) \Omi
$$
satisfy $\overline{\Gamma_{\text{in}}(b)} \subset \Gamma_{\text{out}}(a)$. 
It then follows as in \cite[Section~3]{AmatoGavitone} that there exist $\bar{a}, \bar{b} \in (0,1)$ such that $\Gamma_{\text{out}}(\bar{a}) \setminus \overline{\Gamma_{\text{in}}(\bar{b})} \in \mathcal{K}_{\alpha,\beta}^N$ for some $\alpha<\beta$. 

\subsection{Counterexample to Theorem~\ref{thm:FK} outside \texorpdfstring{$\mathcal{K}_{\alpha,\beta}^N$}{K}}\label{sec:counterexample}

Take any $N \geq 2$ and $\alpha \in (0,1)$.
Consider the domain $\Omega_k = R_k \setminus \overline{B_\alpha}$, where $R_k = (-1,1)^{N-1} \times (-k,k)$ is the rectangular parallelepiped with $k>1$, see Figure~\ref{fig:Om1}. Clearly, we have $\overline{B_\alpha} \subset  R_k$. 
Let $\beta_k>0$ be such that 
\begin{equation}\label{eq:pbprx}
|\partial B_{\beta_k}| = |\partial R_k|.
\end{equation}
As discussed in Section~\ref{sec:K} (see \eqref{eq:badK}), we always have
$$
|R_k \setminus \overline{B_\alpha}| < |B_{\beta_k} \setminus \overline{B_\alpha}|.
$$
That is, $\Omega_k$ does not belong to the class $\mathcal{K}_{\alpha,\beta_k}^N$, but it does satisfy the quermassintegral and perimeter constraints appearing in the definition of $\mathcal{K}_{\alpha,\beta_k}^N$. 

For any $k>1$, there exists $x_k \in \mathbb{R}^N$ such that the ball $B_{r_k}(x_k)$ of radius $r_k = (\beta_k-\alpha)/2$ centered at $x_k$ is a subset of the (concentric) spherical shell $B_{\beta_k} \setminus \overline{B_\alpha}$. 
At the same time, it is not hard to see from \eqref{eq:pbprx} that $\beta_k \to +\infty$ as $k \to +\infty$.
The domain monotonicity of the Dirichlet eigenvalues then implies 
\begin{equation}\label{eq:lb}
\lambda_1(B_{\beta_k} \setminus \overline{B_\alpha}) < \lambda_1(B_{r_k}(x_k)) 
\equiv 
\lambda_1(B_{r_k})
\to 0 
\quad \text{as}~ k \to +\infty. 
\end{equation}
On the other hand, using the well-known expression for the Dirichlet eigenvalues of parallelepipeds, we get
$$
\lambda_1(R_k) = \frac{\pi^2}{4}
\left(
N-1 + \frac{1}{k^2}
\right)
\to 
\frac{\pi^2}{4} (N-1)
\quad \text{as}~ k \to +\infty. 
$$
Again in view of the domain monotonicity, we have
\begin{equation}\label{eq:lr}
\lambda_1(\Omega_k) 
\equiv
\lambda_1(R_k \setminus \overline{B_\alpha})
> 
\lambda_1(R_k) 
\geq
\frac{\pi^2}{4} (N-1) + \rho(k),
\end{equation}
where $\rho(k) \to 0$ as $k \to +\infty$. 
Comparing \eqref{eq:lb} and \eqref{eq:lr}, we conclude that, for any sufficiently large $k$, 
$$
\lambda_1(\Omega_k) > \lambda_1(B_{\beta_k} \setminus \overline{B_\alpha}),
$$
which is the reverse of the Hersch--Weinberger inequality \eqref{eq:reverseFK} from Theorem~\ref{thm:FK}. 

Notice that the domain $\Omega_k$ is merely Lipschitz, but it can be regularized using the well-known stability of Dirichlet eigenvalues. 
Moreover, in view of the convergence of the Robin eigenvalues to the Dirichlet ones as the Robin parameter tends to $+\infty$ (cf.\ \cite[Proposition~4.5]{bucur2017robin}), we conclude that the same result holds at least when $\hi$ and $\ho$ are sufficiently large.  

\begin{figure}[!ht]
    \centering
    \includegraphics[width=0.5\linewidth]{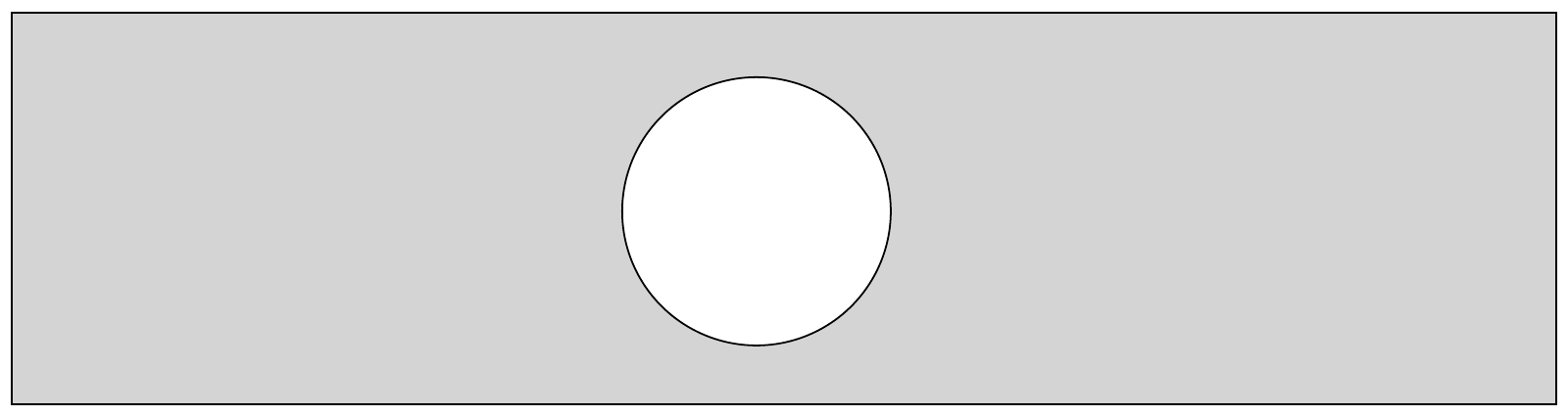}
    \caption{The domain $\Omega_k$ in $\mathbb{R}^2$.}
    \label{fig:Om1}
\end{figure}

\section{Proof of Theorem~\ref{thm:FK}}\label{sec:proof_FK}

Let $\Omega \in \mathcal{K}_{\alpha,\beta}^N$ and $\hi,\ho \in (0,+\infty]$.
Throughout this section, we denote by $(\lambda_1(\Omega), u)$ the first eigenpair of \eqref{eq:D}, where $u>0$ in $\Omega$.  
Let us outline the arguments.
In Section~\ref{sec:approx:morse}, we show that $u$ can be approximated by the first eigenfunctions of perturbed problems in such a way that the approximating first eigenfunctions have no degenerate critical points.   
Then, in Section~\ref{sec:approx:cut}, we analyze the gradient vector field of approximating first eigenfunctions and prove Theorem~\ref{thm:FK}.

\subsection{Approximation by Morse eigenfunctions}\label{sec:approx:morse}

Our aim is to prove that $u$ can be approximated by the first eigenfunctions of a family of perturbed problems, which are Morse functions.
In the pure Dirichlet case $\hi,\ho=+\infty$, 
this can be achieved by perturbing the domain $\Omega$, see \cite{uhl}, but we are not aware of such a result in the non-Dirichlet case. 
Because of that, we provide an alternative construction in which we perturb the problem \eqref{eq:D} by adding a potential.
Our approach is inspired by \cite[Section~7]{albert1978generic}. 

Consider the following eigenvalue problem with a potential $V \in L^\infty({\Omega})$: 
\begin{align}\tag{$\mathcal{RR}_V$}\label{cutproblem2d-pert}
	\left\{\begin{aligned}
		-\Delta v + V(x) v &= \lambda v \quad\text{in} \quad\Omega,\\
		\frac{\pa v}{\pa \nu}+ \hi v&=0 \quad\;\;\text{on}\quad  \pa \Omi,\\
	\frac{\pa v}{\pa \nu}+ \ho v &=0 \quad\;\;\text{on}\quad  \pa \Omo.
	\end{aligned}
	\right.
\end{align}
Analogously to \eqref{eq:lambda1}, the first eigenvalue $\lambda_1(\Omega, V)$ of \eqref{cutproblem2d-pert} can be characterized as
\begin{equation}\label{eq:lambda1-pert}
	\lambda_1(\Omega, V) 
	=
	\inf_{v \in \widetilde{H}^1(\Omega)}
	\frac{\int_{\Om}|\nabla v|^2\dx+ \int_{\Om}V(x) v^2\dx + \ho \int_{\pa \Omo} v^2\dsi+\hi \int_{\pa \Omi} v^2\dsi}{\int_{\Om} v^2\dx},
\end{equation}
where $\widetilde{H}^1(\Omega)$ is as defined in \eqref{eq:H1}.
We directly observe the following universal bounds:
\begin{equation}\label{eq:lambda1-pert:1}
	\lambda_1(\Omega) - \|V\|_\infty 
    \leq
    \lambda_1(\Omega, V) 
	\leq
    \lambda_1(\Omega)
    +
    \|V\|_\infty.
\end{equation}

Before stating the main result of this subsection (Proposition~\ref{prop:approx}), we introduce some notation and prove a technical lemma. We define 
$$
\vertiii{\xi}_k = \sup_{x \in \overline{\Omega}} \sum_{|\alpha| \leq k} |D^\al \xi(x)|
\quad \text{for}~ \xi\in C^k(\overline{\Omega}).
$$
Recall that $(\lambda_1(\Omega), u)$ is the first eigenpair of \eqref{eq:D}, where $u>0$ in $\Omega$.
Let $\mathcal{C}$ be the critical set of $u$, that is,
\begin{equation}\label{eq:critdef}
\mathcal{C} 
= 
\{x \in \overline{\Omega}:~ |\nabla u(x)| = 0\}.
\end{equation}
Observe that there exists a compact set $K \subset \Omega$ such that $\mathcal{C}\subset {\rm Int} (K)$. 
Indeed, if no such $K$ exists, then, by the regularity of $u$ up to the boundary, there would be a critical point $x_0 \in \pa\Om$. 
However, in view of the boundary conditions, we get $u(x_0)=0$, which, combined with $u>0$ in $\Omega$ and the regularity of the boundary, 
contradicts the boundary point lemma, cf.\ \cite[Lemma~4.1]{ABG2}. 

Let us take any cut-off function $\varphi \in C_c^\infty(\Omega)$ with the following  properties:
\begin{equation}\label{eq:phi-properties}
0 \leq \varphi \leq 1 ~\text{in}~ \Omega
\quad \text{and} \quad 
\varphi = 1 ~\text{in}~ K.
\end{equation}
Let $S= {\rm supp}(\varphi)$, so that $K \subset S$. 
Denote
\begin{align}\label{def:gammas}
\gamma_{1} &= \min\{u(x):~ x \in S\} > 0,\\
\gamma_{2} &= \inf\{|\nabla u(x)|:~ x \in \Omega \setminus K\} > 0,\\ 
\gamma_{3} &= \sup\{|\nabla \varphi(x)|:~ x \in \Omega \setminus K\} > 0,
\end{align}
and 
\begin{equation}\label{def:gamma3}
\gamma 
= 
\min
\left\{
\gamma_{1}, \gamma_{2}, \frac{2\gamma_2}{\gamma_3+1} 
\right\} 
> 0.
\end{equation}
For any $v \in C^1(\overline{\Omega})$ satisfying $\vertiii{v-u}_1 < \gamma/2$,  we set
\begin{equation}\label{def:un}
w 
= \varphi  v+
(1-\varphi) u  
\equiv
u - \varphi (u-v).
\end{equation}

We now provide a technical lemma. 
\begin{lemma}\label{lem:uv}
Let $v$ and $w$ be as defined above. 
Then the following assertions hold:
\begin{enumerate}[label={\rm(\roman*)}]
	\item\label{lem:uv:1}  $w = v$ in $K$, and $w=u$ in $\Omega\setminus S$;
    \item \label{lem:uv:2}  the critical sets of $w$ and $v$ coincide and are contained in $K$; 
	\item\label{lem:uv:3} $w>0$ in $\Omega$ and $w > \gamma/2$ in $S$. 
\end{enumerate}
\end{lemma}
\begin{proof}
	\ref{lem:uv:1} is evident from the definition \eqref{def:un} of $w$ and the properties \eqref{eq:phi-properties} of $\varphi$. 
    
	\noindent ~\ref{lem:uv:2}
	Since $\vertiii{v-u}_1 < \gamma/2$, we get
	$$
	-\gamma/2 < |\nabla v| - |\nabla u| < \gamma/2
	\quad \text{in}~ \Omega.
	$$
	Therefore, by \eqref{def:gammas} and \eqref{def:gamma3}, 
	$$
	|\nabla v| > |\nabla u| - \gamma/2 \geq \gamma_2 - \gamma/2>0
	\quad \text{in}~
	\Omega \setminus K.
	$$
	On the other hand, we have
	\begin{equation}\label{def:un2}
	\nabla w 
	=
	\nabla u - \nabla \varphi \, (u-v) -\varphi (\nabla u - \nabla v).
	\end{equation}
    Since $\vertiii{v-u}_1 < \gamma/2$, using \eqref{eq:phi-properties} and \eqref{def:gamma3}, we deduce that
	\begin{align}
		|\nabla w| 
		&\geq 
		|\nabla u|
		- |\nabla \varphi| \, |u-v|
		-
		|\nabla u - \nabla v| 
		\\
		&>
		\gamma_2 
		-
		\gamma_3 \cdot \gamma/2 
		- 
		\gamma/2
		\geq 0
		\quad \text{in}~
		\Omega \setminus K.
	\end{align}
    Thus, both $ w$ and $ v$ do not have critical points in $\Omega \setminus K$. Since $w=v$ in $K$, the assertion~\ref{lem:uv:2} follows. 
	
    \noindent ~\ref{lem:uv:3} Using $\vertiii{v-u}_1 < \gamma/2$ and \eqref{eq:phi-properties}, we have $w = u > 0$ in $\Omega\setminus S$, and 
	$$
	w \geq u - |u-v| 
	> 
	\gamma_1 - \gamma/2 \geq \gamma/2 > 0
	\quad \text{in}~ S.
	$$
    The proof is complete.
\end{proof}

The main result of this subsection is the following proposition. 
\begin{proposition}\label{prop:approx}
Let $\hi,\ho \in (0,+\infty]$. Let $(\lambda_1(\Omega), u)$ be the first eigenpair of \eqref{eq:D}.  
Then, for any $n \in \mathbb{N}$, there exists $V_n \in C_c(\Omega)$ such that the first eigenpair $(\lambda_1(\Omega,V_n), u_n)$ of \eqref{cutproblem2d-pert} and $V_n$ satisfy the following assertions:
    \begin{enumerate}[label={\rm(\roman*)}]
    	\item\label{prop:approx:1} $\lambda_1(\Omega,V_n) = \lambda_1(\Omega)$;
        \item\label{prop:approx:1.5} 
        $u_n > 0$ in $\Omega$ and 
        $u_n \in C^{2,\theta}(\overline{\Omega})$;
        \item\label{prop:approx:2} $u_n$ has only non-degenerate critical points (i.e., $u_n$ is a Morse function);
        \item\label{prop:approx:3} 
        $u_n$ has no minimum points in $\Omega$;
    	 \item\label{prop:approx:4} $\|V_n\|_\infty\rightarrow 0$ and $\vertiii{u_n - u}_2 \rightarrow 0$ as $n\rightarrow +\infty$.
	\end{enumerate}
\end{proposition}
\begin{proof}
For $a\in \R^N\setminus\{0\}$, consider the function $v(x)=u(x)+a\cdot x$. 
Clearly, $x$ is a  critical point of $v$ if and only if $\nabla u(x)=-a$. 
Moreover, $x$ is a degenerate critical point of $v$ if and only if  $-a$ is critical value of $\nabla u$.  
By Sard's theorem, it is known that $\{a\in \mathbb{R}^N: -a \text{ is a critical value of } \nabla u \}$ is of Lebesgue measure zero. Thus, we can choose a sequence $\{a_n\}$ in $\R^N\setminus\{0\} $ that converges to zero such that all critical points of $v_n(x)=u(x)+a_n\cdot x$ are non-degenerate. 
In particular, each $v_n$ is a Morse function in $\Omega$. 
Now, for every $n \in \mathbb{N}$, we define 
$$
u_n(x) =\varphi(x)v_n(x)+(1-\varphi(x))u(x)= u(x)+\varphi(x)(a_n\cdot x),
$$ 
where $\varphi \in C_c^\infty(\Omega)$ is as given in \eqref{eq:phi-properties}. 
Clearly,  $u_n \in C^{2,\theta}(\overline{\Omega})$. 
We can also start indexing $\{a_n\}$ so that 
$$
\vertiii{u_n-u}_1=\vertiii{\varphi (x)(a_n\cdot x)}_1< \frac{\gamma}{2}
\quad \text{for any}~ n \in \mathbb{N}.
$$
Therefore, by Lemma~\ref{lem:uv}, $u_n>0$ in $\Omega$, and the critical sets of $u_n$ and $v_n$ coincide and belong to $K$, so that $u_n$ is also a Morse function.
This proves 
 the assertions~\ref{prop:approx:1.5} and \ref{prop:approx:2}. 

Let us consider the assertion~\ref{prop:approx:3}. 
Denote $\gamma_4:=\la_1(\Omega)\min_{y\in K} u(y) > 0$. 
We can further omit a finite number of elements from $\{a_n\}$
so that 
$$
|\Delta u_n(x)+\la_1(\Omega)u(x)|
=
|\Delta u_n(x)- \Delta u(x)|=|\Delta(\varphi (x)(a_n\cdot x))|< \frac{\gamma_4}{2}
\quad \text{for any}~ x\in \Omega ~\text{and}~ n\in \mathbb{N},
$$
and hence
$$
\Delta u_n<-\la_1(\Omega)u+\frac{\gamma_4}{2}<-\frac{\gamma_4}{2} \quad \text{in}~ K.
$$ 
Thus, critical points of $u_n$ in $K$ are not minimum points. 
Since $u_n$ does not have critical points in $\Omega \setminus K$, we conclude that $u_n$ does not have minimum points in $\Omega$ for every $n\in\mathbb{N}$.

Next, we show that $u_{n}$ is the first eigenfunction of \eqref{cutproblem2d-pert} with an appropriate potential $V_{n}$. 
Let us define
\begin{equation}\label{eq:potV}
V_{n} = \frac{\Delta u_{n} + \lambda_1(\Omega) u_{n}}{u_{n}}.
\end{equation}
Since $u_{n}=u$ in $\Omega\setminus S$, we have $V_{n} = 0$ in $\Omega\setminus S$. 
As $u_n> \gamma/2$ in $S$ by Lemma~\ref{lem:uv}, we get $V_{n} \in C_c^(\Omega)$. 
Clearly,
$$
-\Delta u_{n}+V_n u_n=\lambda_1(\Omega) u_{n}
\quad \text{in}~ \Omega.
$$
Again by the equality $u_n=u$ in $\Omega\setminus S$, we see that $u_n$ satisfies the same Robin boundary conditions on $\partial\Omega$ as $u$. 
Thus, $\lambda_1(\Omega)$ is an eigenvalue of \eqref{cutproblem2d-pert}. 
Recalling the positivity of $u_n$ in $\Omega$ and noting that the only sign-constant eigenfunction of \eqref{cutproblem2d-pert} is the first one, we conclude that $\lambda_1(\Omega,V_n) = \lambda_1(\Omega)$. 
Moreover, we obtain
\begin{align*}
	|V_{n}| 
	= 
	\frac{|(\Delta + \lambda_1(\Omega))(u_{n}-u)|}{|u_{n}|}
	\leq 
	\frac{|\Delta(u_{n}-u)|}{|u_{n}|} 
	+
	\lambda_1(\Omega) \frac{|u_{n}-u|}{|u_{n}|}. 
\end{align*} 
Therefore, 
recalling that $V_{n} = 0$ in $\Omega\setminus S$ and $u_n> \gamma/2$ in $S$, we have 
$$
\|V_{n}\|_\infty 
\leq 
\frac{2}{\gamma}(1 + \lambda_1(\Omega))
\vertiii{u_{n} - u}_{2}.
$$ 
Since $a_n\rightarrow 0$, we get $\vertiii{u_n-u}_2\rightarrow 0$, and hence $\|V_{n}\|_\infty \to 0$ as $n \to +\infty$.
This completes the proofs of the assertions~\ref{prop:approx:1} and \ref{prop:approx:4}. 
\end{proof}

\subsection{Approximation by the gradient flow}\label{sec:approx:cut}

In this subsection, we prove Theorem~\ref{thm:FK}. 
We start with preliminary definitions and several auxiliary results. 

For a fixed $n \in \mathbb{N}$, consider the first eigenfunction $u_n$ of the perturbed problem \eqref{cutproblem2d-pert} with $V=V_n$ given by Proposition~\ref{prop:approx}, and the corresponding gradient-descent system:
	\begin{equation}\label{eq:cauchy}
		\partial_t{\phi}(t,x_0) = -\nabla u_n(\phi(t,x_0)), \quad t \in I,
		\qquad \phi(0,x_0) = x_0 \in \overline{\Omega}, 
	\end{equation}
	where $I = I(x_0) \subset \mathbb{R}$ is the maximal interval containing zero on which $\phi(t,x_0) \in \overline{\Omega}$. 
    Let us emphasize that $\phi$ depends on $n$, but since we assume $n$ to be fixed for most of this subsection, we omit writing the corresponding subscript for $\phi$. 
    The solution $t \mapsto \phi(t,x_0)$ of \eqref{eq:cauchy} is called the \textit{gradient flow line}, or \textit{trajectory}, of $u_n$ through $x_0$. 
    Moreover, the whole trajectory containing $x_0$ is commonly called the \textit{orbit} of $x_0$, and we will sometimes denote it as $\mathcal O_{x_0}$, that is,
\begin{equation}\label{eq:orbit}
\mathcal O_{x_0}
=
\bigcup\limits_{t\in I}\phi(t,x_0). 
\end{equation}
   If $|\nabla u_n(x_0)| > 0$, then $u_n$ is strictly decreasing along the corresponding trajectory, that is, $u_n(\phi(t_1,x_0)) > u_n(\phi(t_2,x_0))$ for all $t_1<t_2$ in $I(x_0)$. 

By the Robin boundary conditions with $\hi,\ho \in (0,+\infty]$ and the boundary point lemma, it follows that $\phi(t,x_0) \in \Omega$ for any $t<0$ and $x_0 \in \overline{\Omega}$, cf.\ \cite[Lemma~4.4]{ABG2}.
Consequently, we have $\{t < 0\} \subset I(x_0)$ for any $x_0 \in \overline{\Omega}$, and hence $\phi(t,\Omega) \subset \Omega$ for $t<0$. 
In particular, $\phi(t,\pa\Omi) \subset \Omega$ and $\phi(t,\pa\Omo) \subset \Omega$ for any $t<0$.  
Moreover, since $u_n \in C^{2,\theta}(\overline{\Omega})$, $\phi(t,\cdot)$ defines a $C^{1,\theta}$-diffeomorphism between $K$ and $\phi(t,K)$ for any subset $K \subset \overline{\Omega}$ and $t<0$.

Let us also observe that since $u_n$ is a Morse function, $\lim\limits_{t \to -\infty} \phi(t,x_0)$ exists, belongs to $\Omega$, and is a critical point $u_n$ for any $x_0 \in \overline{\Omega}$, cf.\ \cite[Lemma~7.4.4 or 7.4.7]{jost2005riemannian}. 
In the same way, if $\phi(t,x_0) \in \overline{\Omega}$ for all $t > 0$, then $\lim\limits_{t \to +\infty} \phi(t,x_0)$ exists, belongs to $\Omega$, and is a critical point $u_n$.
    
\begin{remark}\label{rem:wein}
	In \cite{wein}, \textsc{Weinberger} studied the planar case $N=2$ with $\hi,\ho = +\infty$ and considered the set
\begin{equation}\label{eq:G}
	G = \{x\in \Omega:~ \phi(t_x, x) \in \partial \Omi\; \text{for some}\;
	t_x \in \mathbb{R}\},
	\end{equation}
    where $\phi$ is the solution of \eqref{eq:cauchy} for the first eigenfunction $u$ of \eqref{eq:D}. 
    It was proved that $G$ is open and there is a connected component $\widetilde{\gamma}$ of $\partial G \cap \Omega$ consisting of a finite number of analytic curves, each of which is either a curve of critical points of $u$, or a flow line ${\phi}(\cdot,x)$. Moreover, 
    $\partial u/ \partial \nu = 0$ a.e.\ on $\widetilde{\gamma}$, and  
    $\widetilde{\gamma}$ separates $\partial \Omi$ from $\partial \Omo$ in the sense that any curve lying in $\Omega$ and connecting $\partial \Omi$ with $\partial \Omo$ crosses $\widetilde{\gamma}$. 
	\textsc{Weinberger} called $\widetilde{\gamma}$ the \textit{effectless cut}, since $\Omega$ might be ``cut'' in two annular domains along $\widetilde{\gamma}$ without lowering its first eigenvalue, and this concept was used by \textsc{Hersch} \cite{hersh} to prove the inequality \eqref{eq:reverseFK} as described in Section \ref{sec:intro}; see \cite{ABG2} for more details and a revision of \cite{hersh,wein}. We also refer to \cite{ABG, band2016topological, band2020defining} for related developments in a broader context of investigating the so-called \textit{Neumann domains} of Laplace eigenfunctions. 
  
The same set $G$ and the corresponding effectless cut $\widetilde{\gamma}$ (as a part of $\partial G \cap \Omega$) can be defined in a similar way in higher dimensions $N \geq 3$. 
However, as we mentioned in Section~\ref{sec:intro}, in this case, the analysis of the regularity of $G$ becomes very complicated, and we also refer to Section~\ref{sec:surface} for further discussion. 
To overcome this difficulty, we consider ``approximations of $G$'' along the gradient flow, rather than $G$ itself.
\end{remark}

In what follows, in accordance with the notation in \eqref{eq:lllll}, we use analogous notation for the first eigenvalue of \eqref{cutproblem2d-pert} with $V=V_n$:
$$
\lambda_1^{\mathcal{RR}}(\Om, V_n),~ 
\lambda_1^{\mathcal{NR}}(\Om, V_n),~ 
\lambda_1^{\mathcal{RN}}(\Om, V_n),
$$
where the superscript $\mathcal{RR}$ 
corresponds to the Robin-Robin case $\hi,\ho \in (0,+\infty]$, $\mathcal{NR}$ corresponds to the Neumann-Robin case $\hi=0$, $\ho \in (0,+\infty]$, and $\mathcal{RN}$ corresponds to the Robin-Neumann case $\hi \in (0,+\infty]$, $\ho =0$.

\begin{figure}[!ht]
    \centering
\includegraphics[width=0.4\linewidth]{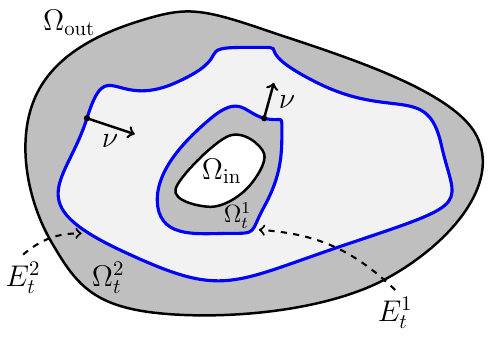}
\captionsetup{width=0.75\textwidth}
    \caption{The sets $E^1_t,E^2_t$, and the corresponding domains $\Om^1_t,\Om^2_t$. Here, $\nu$ is the unit outward normal vector to $\Om^1_t$ and $\Om^2_t$.}
    \label{fig:Omega_t}
\end{figure}

For a fixed $n \in \mathbb{N}$ and any $t<0$, consider two sets 
(see Figure \ref{fig:Omega_t}):
\begin{align}
    E_t^1 := \phi(t,\pa \Omi)
    \quad \text{and} \quad 
    E_t^2 := \phi(t,\pa \Omo),
\end{align}
and consider two domains
\begin{align}
    \Omega_t^1 := \bigcup_{s \in (t,0)}E_s^1  
    \quad \text{and} \quad  
    \Omega_t^2 := \bigcup_{s \in (t,0)}E_s^2.
\end{align}
Let us collect a few general properties of $\Omega_t^1$ and $\Omega_t^2$. 
\begin{lemma}\label{lem:omegai-properties}
    For $i=1,2$ and $t<0$, we have 
    $\Omega_t^i \subset \Omega$, 
    $\Omega_t^i$ is of class $C^{1,\theta}$, 
    $\Omega_{t_1}^i \subset \Omega_{t_2}^i$ for any $t_2 < t_1 < 0$, and $\overline{\Omega_t^1} \cap \overline{\Omega_t^2} = \varnothing$.   
    Moreover, for each $t<0,$ we have  $$\frac{\partial u_n}{\partial \nu} > 0 \quad \text{on}~ E_t^i, \quad i=1,2.
    $$
    Here, $\nu$ stands for the outward unit normal vector to the boundary of $\Omega_t^i$.
\end{lemma}

\begin{proof}
The inclusions and other assertions about $\Om^i_t$ follow easily from the properties of $\phi$, see the discussion before Remark~\ref{rem:wein}. 
Let us determine the sign of the normal derivative on $E_t^i$. 
Since $\Omi$ is $C^{2,\theta}$-smooth, the Robin boundary conditions and the boundary point lemma give
\begin{equation}\label{eq:normalderivative1}
\frac{\partial u_n}{\partial \nu}(x) < 0
\quad\text{for any}~ x \in \pa\Omi. 
\end{equation}
Hence, at each point $x\in \pa\Omi$ the tangent space $T_x \mathbb R^N$ can be decomposed as the direct sum of $(N-1)$- and $1$-dimensional spaces: 
$$
T_x\mathbb R^N=A_{x,0}\oplus B_{x,0},
$$ 
where 
$$
A_{x,0} :=T_x \pa\Omi 
\cong \mathbb{R}^{N-1}
\quad \text{and} \quad 
B_{x,0} :=T_x\mathcal O_x
\cong \mathbb{R},
$$ 
and where $\mathcal O_x$ is the orbit of $x$, see \eqref{eq:orbit}.
That is, any vector $v_0\in T_x\mathbb R^N$ can be uniquely represented as the direct sum of vectors $a_0\in A_{x,0}$ and $b_0\in B_{x,0}$. 
Set 
$$
A_{x,t} :=T_{\phi(t,x)}E^1_t 
\cong \mathbb{R}^{N-1}
\quad \text{and} \quad 
B_{x,t} :=T_x\mathcal O_{\phi(t,x)}
\cong \mathbb{R}.
$$
Since 
$$
\phi(t,\pa\Omi)=E_t^1
\quad \text{and} \quad 
\phi(t,\mathcal O_x)=\mathcal O_{\phi(t,x)},
$$ 
the differential $D_x\phi(t,x):T_x\mathbb R^N\to T_{\phi(t,x)}\mathbb R^N$ at a point $x \in \pa\Omi$ has the following property: 
$$
D_x\phi(t,A_{x,0})=A_{x,t}
\quad \text{and} \quad 
D_x\phi(t,B_{x,0})=B_{x,t}.
$$
Since $\phi(t,\cdot)$ is a diffeomorphism between $\pa\Omi$ and $E_t^1$, 
$D_x\phi(t,x)$ is an invertible linear mapping, and hence any vector $w_t\in T_{\phi(t,x)}\mathbb R^N$ can be uniquely represented as the direct sum of vectors $a_t\in A_{x,t}$ and $b_t\in B_{x,t}$, that is, 
$$
T_{\phi(t,x)}\mathbb R^N=A_{x,t}\oplus B_{x,t}.
$$ 
Thus, the orbit $\mathcal{O}_x$ transversally (i.e., not tangentially) intersects $E^1_t$ for any $t<0$, and hence a unit normal vector $\nu_{\phi(t,x)}$ to $E^1_t$ at the point $\phi(t,x)$ is not orthogonal to the vector $\nabla u_n(\phi(t,x))$, that is, 
\begin{equation}\label{eq:scalprod1}
\frac{\partial u_n}{\partial \nu_{\phi(t,x)}}(\phi(t,x))
\equiv 
\nabla u_n(\phi(t,x))
 \cdot 
 \nu_{\phi(t,x)}
 \neq 0. 
\end{equation}
To fix the direction of $\nu_{\phi(t,x)}$, let us assume that its direction is adjusted to $\nu_{\phi(0,x)}$, which we choose to be  the \textit{outward} unit normal vector to the boundary of $\Omega$, coinciding with our previously imposed convention. 
By continuity, the sign of the scalar product \eqref{eq:scalprod1} is always the same on $E^1_t$ for any $t\leq 0$. 
Recalling \eqref{eq:normalderivative1}, we conclude that 
$$
\frac{\partial u_n}{\partial \nu_{\phi(t,x)}}(\phi(t,x)) < 0\;\text{for any}\; x \in \pa\Omi \quad\text{and}~ t\leq 0.
$$
Noting now that $\nu_{\phi(t,x)}$ is the \textit{inward} unit normal vector to the boundary $E_t^1$ of $\Omega_t^1$, we conclude that
$$
\frac{\partial u_n}{\partial \nu}(x) > 0
\quad\text{for any}~ x \in E_t^1 \;\text{and}\; t < 0,
$$
where, as above, $\nu$ denotes the \textit{outward} unit normal vector. 
Exactly the same analysis works for $E_t^2$. 
 \end{proof}

We now prove the following crucial fact.
\begin{lemma}\label{lem:ineq:min:n}
Let $n \in \mathbb{N}$ and $t<0$.
Then 
\begin{equation}\label{eq:ineq-main1:2x}
\lambda_1^\mathcal{RR}(\Omega, V_n) 
<
\min \{
\lambda_1^\mathcal{RN}(\Omega_t^1,V_n),
\lambda_1^\mathcal{NR}(\Omega_t^2,V_n)
\}.
\end{equation}
\end{lemma}
\begin{proof}
Let us assume that $\hi,\ho \in (0,+\infty)$. 
We show that $\lambda_1^\mathcal{RR}(\Omega, V_n) 
<
\lambda_1^\mathcal{RN}(\Omega_t^1,V_n)$. 
A similar inequality for $\Omega_t^2$ can be derived using an analogous sequence of arguments.
Since $\Omega_t^1$ is $C^{1,\theta}$-smooth (see Lemma~\ref{lem:omegai-properties}) and $u_n \in C^{2,\theta}(\overline{\Omega})$ satisfies \eqref{cutproblem2d-pert}, 
we can apply the standard Green's formula on $\Omega_t^1$ to get
\begin{equation}
\label{eq:lem:proof:1}
\int_{\Omega_t^1} 
\nabla u_n \cdot \nabla \xi \dx = \lambda_1^\mathcal{RR}(\Omega,V_n) \int_{\Omega_t^1} u_n \xi \dx+ \int_{E_t^1} \frac{\partial u_n}{\partial \nu} \,\xi \dsi
- \hi \int_{\pa\Omi} u_n \xi \dsi
- \int_{\Omega_t^1} V_n u_n \xi \dx
\end{equation}
for any $\xi \in {H}^1(\Omega_t^1)$.

Let us denote by $(\lambda_1^\mathcal{RN}(\Omega_t^1,V_n), v_t)$  the first eigenpair of \eqref{cutproblem2d-pert} in $\Omega_t^1$ with the Robin-Neumann boundary conditions.
We have $v_t \in C^{1,\theta}(\overline{\Omega_t^1})$ (see, e.g., \cite[Theorem~2]{lieberman1988boundary}) and, without loss of generality, we may assume that $v_t>0$ in $\Omega_t^1$. 
The function $v_t$ satisfies
\begin{equation}
\label{eq:lem:proof:1x}
\int_{\Omega_t^1} \nabla v_t \cdot \nabla \xi \dx = \lambda_1^\mathcal{RN}(\Omega_t^1,V_n) \int_{\Omega_t^1} v_t \xi \dx
- \hi \int_{\pa\Omi} v_t \xi \dsi
- \int_{\Omega_t^1} V_n v_t \xi \dx
\end{equation}
for any $\xi \in {H}^1(\Omega_t^1)$.

Observe that $u_n|_{\Omega_t^1}\in {H}^1(\Omega_t^1)$. 
Taking $\xi = v_t$ in \eqref{eq:lem:proof:1} and $\xi = u_n|_{\Omega_t^1}$ in \eqref{eq:lem:proof:1x}, 
and subtracting, we get
\begin{equation}\label{Compare}
(\lambda_1^\mathcal{RN}(\Omega_t^1,V_n) - \lambda_1^\mathcal{RR}(\Omega,V_n)) \int_{\Omega_t^1} v_t u_n \dx = \int_{E_t^1} \frac{\partial u_n}{\partial \nu} v_t \dsi.
\end{equation}
Since $u_n$ and $v_t$ are positive, the desired inequality \eqref{eq:ineq-main1:2x} follows from Lemma \ref{lem:omegai-properties}.

The case $\hi=+\infty$ and/or $\ho = +\infty$ can be proved in much the same way. 
\end{proof}

Thanks to the regularity of $\Omega_t^i$, the embedding $\widetilde{H}^1(\Omega_t^i) \subset L^2(\Omega)$ is compact, and hence the first eigenvalues $\lambda_1^\mathcal{RN}(\Omega_t^1)$ and $\lambda_1^\mathcal{NR}(\Omega_t^2)$ of \eqref{eq:D} are well-defined and attained. 
It is known from \cite[Theorem 3.1]{ABG2} (see also \cite{PayneWein1961}) for $N=2$, and from 
\cite[Theorem~1.1]{DellaPiscitelli} (see also \cite[Theorem~1.6]{AnoopMrityunjoy}) for $N \geq 3$ that the following estimate holds for some $\alpha_t > \alpha$: 
\begin{equation}
\label{eq:RFK-ND}
\lambda_1^\mathcal{RN}(\Omega_t^1) 
\leq \lambda_1^\mathcal{RN}(B_{\alpha_t} \setminus \overline{B_{\alpha}}),
\quad \text{where}~
 W_{N-1}(\Om_{\text{in}})=W_{N-1}(B_\al)
 ~\text{and}~ 
|\Omega_t^1|=|B_{\alpha_t} \setminus \overline{B_{\alpha}}|. 
\end{equation}
Notice that $\alpha_t$ is decreasing as a function of $t$ on $(-\infty,0)$, as follows from the inclusion $\Om_{t_1}^1\subset \Om_{t_2}^1$ for $t_2<t_1<0$ (see Lemma~\ref{lem:omegai-properties}).

On the other hand, it is known from \cite{PayneWein1961} (see also \cite[Theorem~3.2]{ABG2}) for $N=2$, and from \cite[Theorem~3.1]{PaoliPiscitelli} (see also \cite[Theorem~1.2]{AnoopMrityunjoy}) for $N \geq 3$ that the following estimate holds for some $\beta_t \in (0,\beta)$:
\begin{equation}
\label{eq:RFK-RN}
\lambda_1^\mathcal{NR}(\Omega_t^2) 
\leq \lambda_1^\mathcal{NR}(B_\beta \setminus \overline{B_{\beta_t}}),
\quad \text{where}~ 
|\pa \Om_{\text{out}}|=|\pa B_\be|
~\text{and}~ 
|\Omega_t^2|=|B_\beta \setminus \overline{B_{\beta_t}}|.
\end{equation}
Moreover, $\beta_t$ is increasing as a function of $t$ on $(-\infty,0)$. 

Thanks to the monotonicity of $t \mapsto \alpha_t,\beta_t$, the value $(\beta_t-\alpha_t)$ increases with respect to $t \in (-\infty,0)$ and it is bounded below. 
Thus, $\lim\limits_{t \to -\infty} (\beta_t-\alpha_t)$ exists.

\begin{lemma}\label{lem:sigma-delta}
For  $n \in \mathbb{N}$ and for $t<0$, let $\Omega_t^1,\Omega_t^2$ and $\alpha_t$, $\beta_t$ be as given above.
Then we have
$|\Omega_t^1|
+
|\Omega_t^2|
\to |\Omega|
$
as 
$t \to -\infty$ 
and
\begin{equation}\label{eq:lem:sigma-delta}
\lim_{t \to -\infty} (\beta_t-\alpha_t) \leq 0,
\end{equation}
where equality holds if and only if $|\Omega| = |B_\beta \setminus \overline{B_{\alpha}}|$.
\end{lemma}
\begin{proof}
Recall that $\overline{\Omega_t^1} \cap \overline{\Omega_t^2} = \varnothing$, see Lemma~\ref{lem:omegai-properties}. 
Consider the set
$$
A := \Omega \setminus 
\bigcup_{t<0} (\Omega_t^1 \cup \Omega_t^2).
$$
If $x\in A$, then, by construction, the orbit $\mathcal{O}_{x}$ reaches neither $\partial \Omega_{\text{in}}$ nor 
$\partial \Omega_{\text{out}}$ in $I$. 
Consequently, $\mathcal{O}_{x}$ stays in $\Omega$, and hence in $A$, for all times, so that $I(x) = \mathbb{R}$.
Since the first eigenfunction $u_n$ strictly decreases along the gradient flow lines, and $u_n$ is a Morse function with no minimum points in $\Omega$ by Proposition~\ref{prop:approx}, we deduce that \textit{for any} regular $x \in A$ there exists a saddle point $p$ such that $\lim\limits_{t \to +\infty} \phi(t,x) = p$. 
That is, the corresponding orbit $\mathcal{O}_{x}$ is a part of the stable manifold of $p$ defined as 
\begin{equation}\label{eq:unstable}
W^s_p = \{
x \in \Omega:~ \phi(t,x) \to p ~\text{as}~ t \to +\infty\},
\end{equation}
and we have 
\begin{equation}\label{eq:unstable2x}
A \subset \{x \in {\Omega}:~ |\nabla u_n(x)|=0\} \cup \bigcup_{p ~\text{saddle}} W^s_p.
\end{equation}
By the Stable Manifold Theorem (see, e.g., \cite[p.~107]{perko2013differential}), any such $W^s_p$ is a differentiable manifold of dimension strictly less than $N$, and hence its Lebesgue measure is zero. 
Since the Morse function $u_n$ has finitely many critical points, the right-hand side of \eqref{eq:unstable2x} has zero Lebesgue measure, and hence $A$ must also have zero Lebesgue measure. 
Thus, $|\Omega_t^1|
+
|\Omega_t^2|
\to |\Omega|
$
as 
$t \to -\infty$. 

Let us justify the inequality \eqref{eq:lem:sigma-delta}.
Since, by \eqref{eq:RFK-ND} and \eqref{eq:RFK-RN}, 
$|\Omega_t^1|=|B_{\alpha_t} \setminus \overline{B_{\alpha}}|$
and
$|\Omega_t^2|=|B_\beta \setminus \overline{B_{\beta_t}}|$, respectively, 
the first part of the lemma and our constraint $|\Omega|
\geq 
|B_{\beta} \setminus \overline{B_{\alpha}}|$ imply
\begin{align}
|B_\beta \setminus \overline{B_{\beta_t}}|
+
|B_{\alpha_t} \setminus \overline{B_{\alpha}}|
=
|\Omega_t^2|
+
|\Omega_t^1|
\to 
|\Omega|
\geq 
|B_{\beta} \setminus \overline{B_{\alpha}}|
\quad \text{as}~ t \to -\infty,
\end{align}
and hence $-|B_{\beta_t}|+|B_{\alpha_t}| \geq 0$ at $t \to -\infty$.
This yields \eqref{eq:lem:sigma-delta}.
\end{proof}

\begin{remark}
    The proof of Lemma~\ref{lem:sigma-delta} is the only place in the entire proof of Theorem~\ref{thm:FK} where the Morse property of the first eigenfunction is used. 
    We note that, for a general (even real analytic) function, the stable or unstable manifold of a degenerate saddle point might have the dimension $N$ (see, e.g., \cite[Section~2.4]{ABG} for more details), and hence we cannot conclude \eqref{eq:lem:sigma-delta} in the same way as above. 
    We believe that even in this case the result of Lemma~\ref{lem:sigma-delta} continues to hold, although a proof is currently lacking.  
\end{remark}

Let us now state the following domain monotonicity result for the first eigenvalue of the Laplacian on a spherical shell satisfying mixed boundary conditions. 
See \cite[Lemmas~2.4 and 2.5]{ABG2} for an analogous result in the planar case; the same arguments extend verbatim to higher dimensions.

\begin{lemma}\label{lem:domain_mono}
  The mapping $r \mapsto \lambda_1^\mathcal{RN}(B_r \setminus \overline{B_{\alpha}})$ is continuous and decreasing in $(\alpha,+\infty)$, and the mapping $r \mapsto \lambda_1^\mathcal{NR}(B_\beta \setminus \overline{B_r})$ is continuous and increasing in $(0,\beta)$.  
  Moreover, we have
  $$
  \lambda_1^{\mathcal{RR}}(B_{\beta} \setminus \overline{B_{\alpha}})
=
\max_{\delta \in (\alpha,\beta)}
\min \{
\lambda_1^{\mathcal{RN}}(B_{\delta} \setminus \overline{B_{\alpha}}),
\lambda_1^{\mathcal{NR}}(B_{\beta} \setminus \overline{B_{\delta}})
\}.
  $$
\end{lemma}

Thanks to Lemmas~\ref{lem:sigma-delta} and \ref{lem:domain_mono}, we establish the following result.

\begin{lemma}\label{lem:ineq:min}
For  $n \in \mathbb{N}$ and for $t<0$, let $\alpha_t$, $\beta_t$ be as given above. 
Then 
\begin{equation}\label{eq:an:1}
\min \{
\lambda_1^\mathcal{RN}(B_{\alpha_t} \setminus \overline{B_{\alpha}}),
\lambda_1^\mathcal{NR}(B_\beta \setminus \overline{B_{\beta_t}})
\}
\leq 
\lambda_1^\mathcal{RR}(B_\beta \setminus \overline{B_\alpha})
+\rho_n(t), 
\end{equation}
where $\rho_n(t) \to 0$ as $t \to -\infty$. 
\end{lemma}
\begin{proof}
Assume that there exists $t_0 < 0$ such that $\alpha_{t_0}\ge \beta$ and $\beta_{t_0}\leq \alpha$. 
Then, by Lemma~\ref{lem:domain_mono}, for any $t<t_0$ we have 
$$
\lambda_1^{\mathcal{RN}}(B_{\alpha_t} \setminus \overline{B_{\alpha}}) \leq \lambda_1^{\mathcal{RN}}(B_{\beta} \setminus \overline{B_{\alpha}}) 
\quad \text{and}\quad 
\lambda_1^\mathcal{NR}(B_\beta \setminus \overline{B_{\beta_t}}) \leq \lambda_1^{\mathcal{NR}}(B_{\beta} \setminus \overline{B_{\alpha}}).
$$
From the variational characterization \eqref{eq:lambda1}, we get
$$
\lambda_1^{\mathcal{RN}}(B_{\beta} \setminus \overline{B_{\alpha}})\le \lambda_1^\mathcal{RR}(B_\beta \setminus \overline{B_\alpha})
\quad \text{and}\quad 
\lambda_1^{\mathcal{NR}}(B_{\beta} \setminus \overline{B_{\alpha}})\le \lambda_1^\mathcal{RR}(B_\beta \setminus \overline{B_\alpha}).
$$ 
Thus, we obtain \eqref{eq:an:1} with $\rho_n(t)=0$ for this case. 

Assume that $\alpha_{t}< \beta$ for every $t<0$. 
Then, by Lemma~\ref{lem:domain_mono}, we have 
\begin{align}\label{eq:mono1}
\lambda_1^{\mathcal{RR}}(B_{\beta} \setminus \overline{B_{\alpha}})
\geq
\min \{
\lambda_1^{\mathcal{RN}}(B_{\alpha_t} \setminus \overline{B_{\alpha}}),
\lambda_1^{\mathcal{NR}}(B_{\beta} \setminus \overline{B_{\alpha_t}})
\}.
\end{align}
Consider the following two cases.
 
 \noi \textbf{Case 1:} $|\Omega| = |B_\beta \setminus \overline{B_\alpha}|$. In this case, by Lemma~\ref{lem:sigma-delta}, we have  $\lim\limits_{t\rightarrow -\infty}(\alpha_t- \beta_t)=0$.
 Thus, we can write 
 \begin{align}
 \lambda_1^{\mathcal{NR}}(B_{\beta} \setminus \overline{B_{\alpha_t}}) = \lambda_1^{\mathcal{NR}}(B_{\beta} \setminus \overline{B_{\beta_t}})- \rho_n(t), 
 \end{align}
 where $\rho_n(t)>0$ and $\rho_n(t)\rightarrow 0$ as $t\rightarrow -\infty$.
 Therefore, we get \eqref{eq:an:1}:
 \begin{align}\label{eq:epsilon1}
 \lambda_1^{\mathcal{RR}}(B_{\beta} \setminus \overline{B_{\alpha}})
 \geq
 \min \{
 \lambda_1^{\mathcal{RN}}(B_{\alpha_t} \setminus \overline{B_{\alpha}}),
 \lambda_1^{\mathcal{NR}}(B_{\beta} \setminus \overline{B_{\beta_t}})
 \}-\rho_n(t).
 \end{align}
\noi \textbf{Case 2:} $|\Omega| > |B_\beta \setminus \overline{B_\alpha}|$.  In this case, by Lemma~\ref{lem:sigma-delta}, there exists $t_n<0$ such that 
$\alpha_t> \beta_t$ whenever $t\le t_n$. 
Thus, using again Lemma~\ref{lem:domain_mono},  we obtain 
\begin{align}
\lambda_1^{\mathcal{NR}}(B_{\beta} \setminus \overline{B_{\alpha_t}}) > \lambda_1^{\mathcal{NR}}(B_{\beta} \setminus \overline{B_{\beta_t}})
\quad \text{for any}~
t\leq t_n,
\end{align}
and hence \eqref{eq:an:1} directly follows from \eqref{eq:mono1}. 

The case where $\beta_t > \alpha$ for every $t<0$ can be handled using similar arguments by substituting $\alpha_t$ with $\beta_t$ in \eqref{eq:an:1}.
\end{proof}

\medskip
With the established auxiliary results in hand, we now conclude the proof of Theorem~\ref{thm:FK}. 
\begin{proof}[Proof of Theorem~\ref{thm:FK}]

Thanks to Lemma~\ref{lem:ineq:min:n}, for any $n \in \mathbb{N}$ and $t<0$ we have
\begin{equation}\label{eq:ineq-main1:2}
\lambda_1^\mathcal{RR}(\Omega, V_n) 
< 
\min \{
\lambda_1^\mathcal{RN}(\Omega_t^1,V_n),
\lambda_1^\mathcal{NR}(\Omega_t^2,V_n)
\}.
\end{equation}
From \eqref{eq:lambda1-pert:1}, 
we get the bounds 
\begin{equation}\label{eq:ineq-main1:3}
\lambda_1^\mathcal{RN}(\Omega_t^1,V_n)
\leq
\lambda_1^\mathcal{RN}(\Omega_t^1)
+\|V_n\|_\infty
\quad \text{and} \quad 
\lambda_1^\mathcal{NR}(\Omega_t^2,V_n)
\leq
\lambda_1^\mathcal{NR}(\Omega_t^2)
+\|V_n\|_\infty.
\end{equation}
Recalling that $\lambda_1^\mathcal{RR}(\Omega,V_n) = \lambda_1^\mathcal{RR}(\Omega)$  by Proposition~\ref{prop:approx}, we deduce from \eqref{eq:ineq-main1:2} and \eqref{eq:ineq-main1:3} that
$$
\lambda_1^\mathcal{RR}(\Omega) 
< 
\min \{
\lambda_1^\mathcal{RN}(\Omega_t^1),
\lambda_1^\mathcal{NR}(\Omega_t^2)
\}
+\|V_n\|_\infty.
$$
Combining \eqref{eq:RFK-ND}, \eqref{eq:RFK-RN}, and Lemma~\ref{lem:ineq:min}, for each $n\in \mathbb{N}$ we get
\begin{equation}\label{eq:ineq-main2}
\min \{
\lambda_1^\mathcal{RN}(\Omega_t^1),
\lambda_1^\mathcal{NR}(\Omega_t^2)
\}
\leq 
\lambda_1^\mathcal{RR}(B_{\beta} \setminus \overline{B_{\alpha}})
+
\rho_n(t),
\end{equation} 
where $\rho_n(t) \to 0$ as $t \to -\infty$.
Thus, we arrive at the inequality
$$
\lambda_1^\mathcal{RR}(\Omega) 
\leq
\lambda_1^\mathcal{RR}(B_\beta \setminus \overline{B_\alpha}) 
+\|V_n\|_\infty +
\rho_n(t). 
$$
For a given $\varepsilon>0$, we can choose $n_0\in \mathbb{N}$ so that $\|V_{n_0}\|<\varepsilon,$ and then for this $V_{n_0}$ we choose $t_0 < 0$ so that $\rho_{n_0}(t_0)<\varepsilon$. Thus,
$$
\lambda_1^\mathcal{RR}(\Omega) 
\leq
\lambda_1^\mathcal{RR}(B_\beta \setminus \overline{B_\alpha})
+2\varepsilon. 
$$
Since $\varepsilon$ is arbitrary, we derive the desired Hersch--Weinberger inequality
$$
\lambda_1^\mathcal{RR}(\Omega) 
\leq
\lambda_1^\mathcal{RR}(B_\beta \setminus \overline{B_\alpha}).
\qedhere
$$
\end{proof}

\section{Properties of the effectless cut}\label{sec:surface}
In this section, we propose a definition of the effectless cut in any higher dimension, and we
study some of its general topological and geometrical properties, using the framework of dynamical systems.
As above, $\Omega\subset\mathbb R^N$, $N \geq 2$, is a bounded $C^{2,\tht}$-smooth domain of the form $\Omega = \Omo \setminus \overline{\Omi}$, whose boundary $\partial\Omega$ has two connected components $\pa \Omi$ and $\pa \Omo$. 
Assume that $\overline{\Omi}$ and $\overline{\Omo}$ are $C^2$-diffeomorphic to a closed $N$-dimensional ball, and $\Omega$ is diffeomorphic to $\mathbb S^{N-1}\times (0,1)$. 
Hereinafter, $\mathbb S^k$ stands for the $k$-dimensional sphere. 

 We consider a Morse function $u \in C^{2,\tht}(\overline{\Omega})$ satisfying the following assumptions:
\begin{enumerate}[ label=(\textbf{A\arabic*})]
\item\label{as:u:2} The gradient-descent vector field $-\nabla u$ is transverse to $\partial\Omega$ and points inward along $\partial\Omega$, that is,
$$
\frac{\partial u}{\partial \nu} > 0 
\quad \text{on}~ \pa\Omega.
$$ 
\item\label{as:u:3} $u$ has no maximum points in $\Omega$. 
\end{enumerate}
Clearly, the negative first eigenfunction of the problem \eqref{eq:D} always satisfies the assumptions \ref{as:u:2} and \ref{as:u:3}. 
Moreover, for an appropriate potential $V$, the negative first eigenfunction of the problem \eqref{cutproblem2d-pert} is a Morse functions satisfying \ref{as:u:2} and \ref{as:u:3}, see Proposition~\ref{prop:approx}. 

Since $\mathbb R^N$ is diffeomorphic, via stereographic projection, to $\mathbb S^N$ with one point removed, it will be convenient to regard $\Omega$ as a subset of $\mathbb S^N$. 
In view of our assumptions on $\Omega$ and the assumption \ref{as:u:2}, 
one can extend $u$ to all of $\mathbb S^N$ as a Morse function with exactly two maximum points $\alpha_+$,  $\alpha_-$ such that $\alpha_+ \in \Omi$ and $\alpha_- \in \mathbb{S}^N \setminus \overline{\Omo}$. 
In what follows, we identify $u$ with such an extension and denote by $\phi^t: \mathbb S^N \to\mathbb S^N$, $t \in \mathbb{R}$, the gradient-descent flow of $u$, see \eqref{eq:cauchy} with $\phi^t(\cdot) := \phi(t,\cdot)$.  
Note that $\phi^t$ is a diffeomorphism for each fixed $t$. 

Let us introduce some standard notation from dynamical systems theory and recall some general facts about gradient flows, see, e.g., \cite{grin,Rob,Sm}. 
Let $\mathcal C$ be the set of all critical points of $u$ (or, equivalently, fixed points of $\phi^t$) in $\mathbb{S}^N$. 
For any $p\in \mathcal{C}$, the stable and unstable manifolds of $p$ are defined, respectively, as  
\begin{align}
\label{eq:unstable0}
W^s_p &= \{
x \in \mathbb{S}^N:~ \phi^t(x) \to p ~\text{as}~ t \to +\infty\},\\
\label{eq:stable0}
W^u_p &= \{
x \in \mathbb{S}^N:~ \phi^t(x) \to p ~\text{as}~ t \to -\infty\}.
\end{align}
It follows directly from the definition that if $p, r \in \mathcal{C}$ and $p \neq r$, then
\begin{equation}\label{eq:justref}
W^u_p \cap W^u_r = \varnothing
\quad \text{and} \quad 
W^s_p \cap W^s_r = \varnothing.
\end{equation}
\begin{definition}
    For $p,r\in \mathcal{C},$ we say $W^u_p$ and $W^s_r$ intersect {\it transversely} if either $W^u_p\cap W^s_r= \varnothing,$ or $W^u_p\cap W^s_r\neq \varnothing$ and $T_x W^u_p + T_x W^s_r = T_x \mathbb{S}^N$ for any $x \in W^u_p\cap W^s_r$.

 We say a Morse function $u$ is a \textit{Morse--Smale function} if $W^u_p$ intersects $W^s_r$ transversely for every $p,r\in \mathcal{C}$.
If $p$ and $r$ are saddles, then the intersection $W^u_p\cap W^s_r$ is called a {\it heteroclinic intersection}, see Figure~\ref{AE}. 
\end{definition}

In this section, alongside the assumptions \ref{as:u:2} and \ref{as:u:3}, we make an additional significant assumption regarding the Morse function $u$:
\begin{enumerate}[label=(\textbf{A\arabic*}), start=3] \item\label{as:u:1} $u$ is a Morse--Smale function.
\end{enumerate}

In Proposition~\ref{prop:approx}, we proved that the first eigenfunction of \eqref{eq:D} can always be approximated by the first eigenfunctions of \eqref{cutproblem2d-pert}, which are Morse functions. 
In Proposition~\ref{prop:approx:MS} below, we strengthen this statement by showing that the approximating first eigenfunctions of \eqref{cutproblem2d-pert} can be chosen as Morse--Smale functions. That is, informally speaking, there are ``many'' first eigenfunctions satisfying \ref{as:u:2}, \ref{as:u:3}, and \ref{as:u:1}.

Let us now decompose the set $\mathcal C$ of critical points of $u$ in $\mathbb{S}^N$ as
$$
\mathcal C
=
\mathcal{C}_0 \cup \mathcal{C}_1 \cup \dots \cup \mathcal{C}_N,
$$
where $\mathcal C_k$ is the set of critical points with the Morse index $k$. 
That is, every $p\in\mathcal C_k$ has the unstable manifold $W^u_p$ diffeomorphic to $\mathbb R^k$; in other words, $W^u_p$ is a $k$-cell. 
In the same way, the stable manifold $W^s_p$ is diffeomorphic to $\mathbb R^{N-k}$, that is, $W^s_p$ is a $(N-k)$-cell. 
Points in $\mathcal C_0$ are called {\it sinks} (i.e., minimum points of $u$), points in $\mathcal C_N$ are called {\it sources} (i.e., maximum points of $u$), and all other points in $\mathcal C$ are called {\it saddles}. 
By our assumptions on $u$, we have $\mathcal C_N=\{\alpha_-,\alpha_+\}$. 
Moreover, $W^u_{\alpha_\pm}$ is an $N$-dimensional open (topological) ball. 
We will call $W^u_{\alpha_\pm}$ the \textit{basin} of $\alpha_\pm$.  
For any saddle $p$, every connected component of $W^s_p\setminus \{p\}$ is called a \textit{stable separatrix}, and every connected component of $W^u_p\setminus \{p\}$ is called an \textit{unstable separatrix}. 
In particular, if $p \in \mathcal{C}_{N-1}$, then stable separatrices are one-dimensional and there are exactly two of them, while if $p \in \mathcal{C}_{q}$ with $q \in \{1,\dots,N-2\}$, then $W^s_p\setminus \{p\}$ is connected, and hence there is exactly one stable separatrix. 

By the transversality, if $W^u_p\cap W^s_r\neq\varnothing$ for some $p \in \mathcal{C}_{l}$ and $r \in \mathcal{C}_{k}$, then $l > k$.
In particular, for any $r\in \mathcal{C}_{N-1}$, the one-dimensional set $W^s_r$ is not involved in heteroclinic intersections, i.e., any (stable) separatrix in $W^s_r$ emanates from a source, which is either $\alpha_+$ or $\alpha_-$. 

A closed set $\Lambda \subset \mathbb{S}^N$ is called an \textit{attractor} of the flow $\phi^t$, if $\Lambda$ has a compact neighborhood $U \subset \mathbb{S}^N$ such that $\phi^t(U) \subset \text{int}(U)$ for any $t>0$, and $\Lambda = \bigcap_{t \geq 0} \phi^t(U)$. 
Such $U$ is called a \textit{trapping neighborhood of $\Lambda$}. 
A set is called a \textit{repeller} of $\phi^t$, if it is an attractor of $\phi^{-t}$.  
Any attractor or repeller $\Lambda$ is an $\phi^t$-\textit{invariant set} for any $t \in \mathbb{R}$, i.e., $\phi^t(\Lambda) = \Lambda$. 
The result of \cite[Theorem~1.1]{Glob} says that if, for some subset $\mathcal{C}^* \subset \mathcal{C}$, the union $\bigcup_{r \in \mathcal{C}^*} W^s_r$ is a closed set, then it is necessarily a repeller. 
In the same way, if $\bigcup_{r \in \mathcal{C}^*} W^u_r$ is a closed set, then it is an attractor. 
Since trapping neighborhoods can be chosen as sublevel sets of smooth Lyapunov functions, $U$ can always be taken to be a smooth $N$-dimensional manifold with boundary, and each connected component of $\partial U$ is a hypersurface on $\mathbb{S}^N$.

\smallskip
Let us list a few general basic facts which will be used in the arguments below:
\begin{enumerate}[label = {$(\star)_{\arabic*}$}]
    \item\label{star1} 
    $\mathbb S^N=\bigcup_{p\in\mathcal C}W^u_p=\bigcup_{p\in\mathcal C}W^s_p$, where the unions are disjoint.
    \item\label{star2}
    $W^u_p\cap W^s_r\neq\varnothing$ 
    if and only if $W^u_r\subset\overline{W^u_p}$
    and $W^s_p\subset\overline{W^s_r}$. 
    Moreover,
    $$
    \overline{W^u_p} = \bigcup_{r \in\mathcal C:\, W^u_p \cap W^s_r \neq \varnothing} W^u_r
    \quad \text{and} \quad  
    \overline{W^s_r} = \bigcup_{p \in\mathcal C:\, W^u_p \cap W^s_r \neq \varnothing} W^s_p.
    $$
    \item\label{star3} 
    If $W^u_{\sigma} \cap W^s_p \neq \varnothing$ and $W^u_p \cap W^s_r \neq \varnothing$, then $W^u_{\sigma} \cap W^s_r \neq \varnothing$.
    \item\label{star4} 
    $\overline{W^s_r}$ of any $r \in \mathcal{C}$ contains a source. 
\end{enumerate}

\medskip
For the convenience of further analysis, let us denote
$$
S_+=\pa\Omi,~~
S_-=\pa\Omo, 
\quad \text{and} \quad 
G_\pm =\bigcup\limits_{t > 0,~ x \in S_\pm}\phi^{t}(x).
$$
We also occasionally denote by $A \sqcup B$ the disjoint union of some sets $A$ and $B$, for clarity. 

\medskip
Recall from Remark~\ref{rem:wein} that \textsc{Weinberger} \cite{wein} defined the effectless cut $\widetilde{\gamma}$ as $\partial G \cap \Omega$, where the set $G$ is given by \eqref{eq:G}. 
However, this definition has the drawback that $\widetilde{\gamma}$ can be irregular. 
In particular, $G$ might contain ``cracks'' and ``cusps'', cf.\ \cite{band2020defining}. 
In order to regularize $\widetilde{\gamma}$, in \cite{ABG2}, it was proposed to define the effectless cut as $\partial (\mathrm{Int}(\overline{G})) \cap \Omega$. This set has neither ``cracks'' nor isolated points. 

The analysis performed in the proof of Theorem~\ref{thm:FK} in Section~\ref{sec:approx:cut} suggests another way to define a regularized version of the effectless cut. 
Namely, we define the {\it effectless cut} $E$ for the pair $(\Omega,u)$ as 
\begin{equation}\label{defE}
E=\overline{G_-}\cap\overline{G_+}.    
\end{equation}
The aim of this section is to establish a few fundamental properties of $E$. 

\begin{proposition}\label{prop:efc} 
The effectless cut $E$ for the pair $(\Omega,u)$ has the following properties:
\begin{enumerate}[label={\rm(\arabic*)}]
    \item\label{prop:efc:1} 
$E=\bigcup_{p\in\mathcal C_E}W^u_p$, where $\mathcal{C}_E$ consists of all critical points $r$ for which $\alpha_\pm \in \overline{W^s_r}$. 
    \item\label{prop:efc:3} 
    The topological dimension of $E$ is $N-1$.
    \item\label{prop:efc:4} 
    $E$ is an attractor of $\phi^t$.
    \item\label{prop:efc:2} 
    $E$ is a closed connected subset of $\Omega$ that divides $\Omega$ into two subdomains $\Omega_-$, $\Omega_+$ such that $S_\pm \subset \overline{\Omega_\pm}$. 
\end{enumerate}
\end{proposition}
\begin{proof} 
\ref{prop:efc:1} 
Recall that $\mathcal C_N=\{\alpha_-,\alpha_+\}$, and let $B_\pm := W^u_{\alpha_\pm}$ be the basin of $\alpha_\pm$. 
Observe that $G_\pm \subset B_\pm$ and
$$
B_-\setminus G_- = \mathbb{S}^N \setminus \Omo
\quad \text{and} \quad 
B_+\setminus G_+ = \overline{\Omi}.
$$
Therefore, we have $\overline{B_-\setminus G_-}\cap\overline{B_+\setminus G_+}=\varnothing$, and hence $E=\overline{B_-}\cap\overline{B_+}$. 
We obtain from \ref{star2} that
\begin{equation}\label{eq:bpm1}
\overline{B_\pm} = \bigcup_{\sigma \in\mathcal C:\, B_\pm \cap W^s_\sigma \neq \varnothing} W^u_\sigma.
\end{equation}
Since, by \ref{star1}, for any $x\in E$ there exists $r \in \mathcal{C}$ such that $x\in W^u_r$, we deduce from \eqref{eq:bpm1} that $B_\pm \cap W^s_r \neq \varnothing$ and $W^u_r \subset (\overline{B_-}\cap\overline{B_+}) = E$. 
That is, again in view of \eqref{eq:bpm1}, we conclude that $E=\bigcup_{p\in\mathcal C_E}W^u_p$, where $\mathcal{C}_E$ consists of all critical points $r$ for which $B_\pm \cap W^s_r \neq \varnothing$, and hence $\alpha_\pm \in \overline{W^s_r}$. 

\ref{prop:efc:3}
To show that the topological dimension of $E$ is $N-1$, it suffices to prove that $E$ contains at least one $(N-1)$-cell. 
Denote 
$$
R
=
\mathcal C_N\cup\bigcup\limits_{p\in \mathcal{C}_{N-1}}W^s_p
\equiv
\bigcup\limits_{p\in \mathcal{C}_{N-1} \cup \mathcal{C}_{N}}W^s_p.
$$
It follows from the discussion above that $R$ consists of one-dimensional separatrices together with their end points, see also Figure~\ref{AE}.  
Therefore, $R$ is a one-dimensional closed set, and hence $R$ is a repeller of $\phi^t$ by \cite[Theorem~1.1]{Glob}. 
Moreover, since the ``complementary'' set (which is, in fact, an attractor of $\phi^t$)
$$
A = \bigcup\limits_{p\in \mathcal{C}_{0} \cup \dots \cup \mathcal{C}_{N-2}}W^u_p
$$
has, by definition, the dimension at most $N-2$, we again deduce from \cite[Theorem~1.1]{Glob} that the repeller $R$ is a connected set. 
Consequently, $\mathcal C_{N-1}$ must contain at least one point $p$ whose two one-dimensional separatrices belong to different basins, and we denote the subset of all such points by $\mathcal C^*_{N-1}$. 
Let us show that $W^u_p \subset E$ for any $p\in\mathcal C^*_{N-1}$. 
Indeed, we know from \ref{star2} that  
\begin{equation}\label{eq:ww1}
\overline{W^u_p} 
=
\bigcup\limits_{r \in\mathcal C:\,W^u_p\cap W^s_r\neq\varnothing}W^u_r.
\end{equation}
Recalling that $B_\pm := W^u_{\alpha_\pm}$ and since $B_\pm \cap W^s_p \neq\varnothing$, 
we get from \ref{star3} that 
$B_\pm \cap  W^s_r\neq\varnothing$ 
for any $r \in \mathcal{C}$ such that $W^u_p\cap W^s_r\neq\varnothing$. 
Thus, thanks to \ref{star2}, we obtain  $W^u_r \subset (\overline{B_-}\cap\overline{B_+}) = E$ for any such $r$, which yields $W^u_p \subset E$.
As a consequence, the topological dimension of $E$ is at least $N-1$. 
On the other hand, $E$ cannot have dimension $N$, since the openness of $B_\pm$ implies that $\alpha_\pm \not\in \mathcal{C}_E$.

\begin{figure}[!ht]
    \centering
    \includegraphics[width=0.7\linewidth]{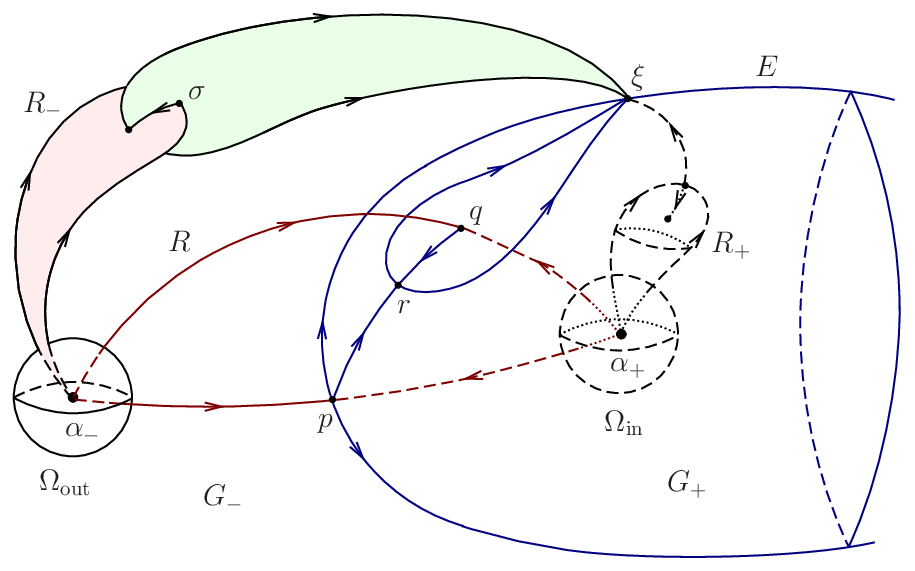}
    \captionsetup{width=0.75\textwidth}
    \caption{Schematic phase portrait of $\phi^t$ with various types of saddles.  
    Intersection between pink and green sets is a heteroclinic intersection. 
    Here, $p,q \in R \cap \mathcal{C}_{N-1}^*$, $\sigma \in R \setminus \mathcal{C}_{N-1}^*$ and $\sigma \in \mathcal{C}_{-}$, and $\xi$ is a sink.} 
    \label{AE}
\end{figure}

\ref{prop:efc:4} 
Since $E$ is closed (by its definition as an intersection of closed sets), the assertion \ref{prop:efc:1} and \cite[Theorem~1.1]{Glob} imply that $E$ is an attractor of $\phi^t$. 

\ref{prop:efc:2}  
Denote 
$$
\mathcal C_+
=
\{p \in \mathcal{C}:~ \overline{W^s_p} \text{ does not contain } \alpha_-\}
\quad \text{and} \quad
\mathcal C_-
=
\{p \in \mathcal{C}:~ \overline{W^s_p} \text{ does not contain } \alpha_+\}. 
$$
Obviously, $\alpha_\pm \in \mathcal C_\pm$, and $\mathcal C_+ \cap \mathcal C_- = \varnothing$ by \ref{star4}.
Let us also consider the sets
$$
R_\pm=\bigcup\limits_{r\in\mathcal C_\pm}W^s_r
\quad \text{and} \quad 
H_\pm=\bigcup\limits_{r\in\mathcal C_\pm}W^u_r.
$$
It is evident from \eqref{eq:justref} that $R_+$ and $R_-$ are disjoint, and, similarly, $H_+$ and $H_-$ are disjoint. 
Moreover, $G_\pm \subset B_\pm \subset H_\pm$. 
It is also clear that $R_\pm$ is connected.  
Let us justify that $R_\pm$ is a closed set. 
For this purpose, it is sufficient to show that $\overline{W^s_r} \subset R_\pm$ for any $r \in \mathcal{C}_\pm$. 
Let us fix some $r \in \mathcal{C}_+$. 
By \ref{star2}, the claim will follow if we prove that $W^s_p \subset R_+$ for any $p$ such that $W^u_p \cap W^s_r \neq\varnothing$. 
If $B_- \cap W^s_p = \varnothing$, then $p \in \mathcal{C}_+$, and the claim follows. 
Suppose now that there exists $p$ such that $B_- \cap W^s_p  \neq \varnothing$ and $W^u_p \cap W^s_r \neq\varnothing$.
But then \ref{star3} yields $B_- \cap W^s_r \neq \varnothing$, which implies that $\alpha_- \in \overline{W^s_r}$, a contradiction to the assumption $r \in \mathcal{C}_+$. 
That is, $R_\pm$ is closed. 

Since $R_\pm$ is closed, \cite[Theorem~1.1]{Glob} implies that $R_\pm$ 
is a repeller of $\phi^t$. Namely, there exists a trapping neighborhood $U_\pm$ of $R_\pm$ such that $R_\pm=\bigcap_{t\leq 0}\phi^t(U_\pm)$, and hence $H_\pm = \bigcup_{t\geq 0}\phi^t(U_\pm)$. 
It then follows from the connectedness of $R_\pm$ that $U_\pm$ can also be chosen connected, and hence $H_\pm$ is connected, as well.  

In view of \ref{star4} and the definitions of $\mathcal{C}_\pm$ and $\mathcal{C}_E$, we see that $\mathcal{C} = \mathcal{C}_+ \sqcup \mathcal{C}_E \sqcup \mathcal{C}_-$. 
Hence, we deduce from \ref{star1} that 
\begin{equation}\label{eq:decompsn}
\mathbb S^N=H_-\sqcup E\sqcup H_+.
\end{equation}
Since $\mathbb S^N\setminus E=H_-\sqcup H_+$ and $H_\pm$ is connected, the set $\Omega_\pm:=H_\pm\cap\Omega$ is the complement of $H_\pm$ to the $N$-dimensional (topological) ball $H_\pm\setminus\Omega$, and hence $\Omega_\pm$ is connected.  
Therefore, $\Omega\setminus E=\Omega_-\sqcup\Omega_+$. 
Moreover, it follows from definition that $E$ is a closed subset of $\Omega$. 
This fact and \eqref{eq:decompsn} imply that $H_\pm$ is open, and hence $\Omega_\pm$ is a domain. 

It remains to prove that $E$ is connected. 
Recall from \ref{prop:efc:4} that 
$E$ is an attractor of $\phi^t$. 
Moreover, in view of \eqref{eq:decompsn},  $U=\overline{\mathbb S^N\setminus(U_+\cup U_-)}$ is its trapping neighborhood. 
Noting that the intersection of nonempty, nested, connected, compact sets is a nonempty, connected, compact set (see, e.g., \cite[Proposition~10.1]{grin}), the connectedness of $E$ will follow from the connectedness of $U$. 
Thus, our aim is to prove that $U$ is connected. 

First, we justify that 
$S^N\setminus U_\pm$ is connected. 
Suppose, by contradiction, that $\mathbb S^N\setminus U_-$ has at least two connected components, say, $T_1$ and $T_2$. 
Clearly, only one of these connected components can contain $\alpha_+$. 
Assume, without loss of generality, that $\alpha_+ \in T_2$, so that $\alpha_\pm \not\in T_1$. 
By definition, $\overline{T_1}$ is a trapping neighborhood for $\phi^{-t}$.  
Consequently, there exists a critical point $r \in T_1$. 
By \ref{star4}, $\overline{W^s_r}$ contains a source. 
That is, there is a continuous path $\gamma$ in $\overline{W^s_r}$, consisting of flow lines, that connects $r$ with either $\alpha_+$ or $\alpha_-$. 
Since $\alpha_\pm \not\in T_1$, $\gamma$ must intersect $\partial T_1 \subset \partial U_-$ and 
enter the trapping neighborhood $U_-$. That is, $\gamma$ necessarily connects $r$ with $\alpha_-$. 
Consequently, we get $r \in R_-$. 
But then, by the definition of $U_-$, we must have $r \in U_-$, contradicting our initial choice $r \in T_1 \subset \mathbb{S}^N \setminus U_-$. 
Thus, we conclude that $\mathbb S^N\setminus U_-$ is connected. 
The same argument shows the connectedness of $\mathbb S^N\setminus U_+$. 

Recalling that each connected component $T$ of $\partial U_\pm$ can be chosen to be a smooth hypersurface, the Jordan-Brouwer separation theorem implies that $T$ divides $\mathbb{S}^N$ into two connected components. 
Since both $\mathbb S^N\setminus U_\pm$ and $U_\pm$ are connected, we conclude that $\partial U_\pm$ must also be connected. 

Since $U_+$ is, by construction, a compact subset of $\mathbb S^N\setminus U_-$, and since both  $\mathbb S^N\setminus U_-$ and $\partial U_+$ are connected, we conclude that 
$\partial U_+$ divides $\mathbb S^N\setminus U_-$ into two connected components $U_+$ and $\mathbb{S}^N \setminus (U_+ \cup U_-)$, see, e.g., \cite[Corollary~3]{lemmens2018local}. 
The connectedness of $\mathbb{S}^N \setminus (U_+ \cup U_-)$ implies that its closure, $U$, is connected. 
Consequently, $E$ is also connected. 
\end{proof}

In the planar case, the topology of the effectless cut $E$ can be refined as follows.
\begin{proposition}\label{prop:smd} 
Let $N=2$. 
Then $E$ is a simple closed curve and hence is homeomorphic to $\mathbb S^{1}$, while  
$\Omega_\pm$ is homeomorphic to $\mathbb S^{N-1}\times(0,1)$. 
\end{proposition}
\begin{proof}
Let $\mathcal C^*_E\subset\mathcal C_E$ be the set of \textit{saddles} belonging to $E$.  
Then $\mathbb{S}^2$ is decomposed by $\bigcup_{p\in\mathcal C_E^*}\overline{W^s_p}$ into $m$ closed (topological) disks $\mathcal{D}_i$, $i \in \{1,\dots,m\}$. 
We assume that the disks $\mathcal{D}_i$ and the saddle points $p_i\in \mathcal C^*_E$ are indexed such that $\mathcal D_i\cap\mathcal D_{i+1}=\overline{W^s_{p_{i+1}}}$, where $\mathcal D_{m+1}:=\mathcal D_1$. 
Let us show that $\mathcal{D}_i \cap E$ contains exactly one sink $\omega_i$. 
(Note that, in general, $\mathcal{D}_i$ might contain several sinks, see Figure~\ref{Bi}.)

\begin{figure}[!ht]
    \centering
    \includegraphics[width=0.6\linewidth]{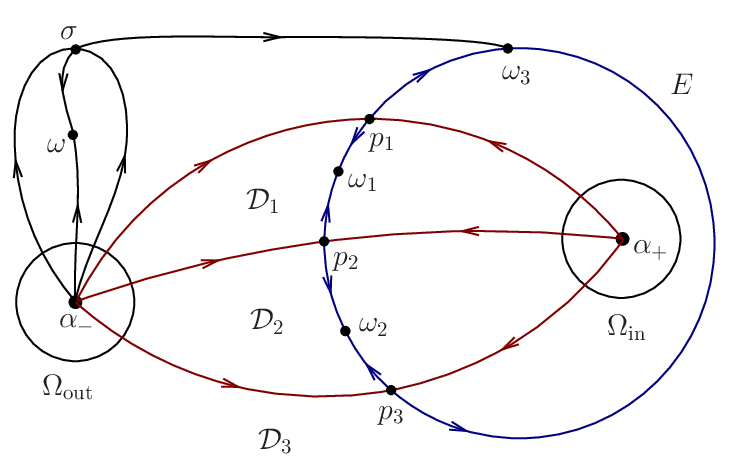}
    \caption{Illustration to Proposition~\ref{prop:smd} with $m=3$.}
    \label{Bi}
\end{figure}

Fix $i$ and consider two consecutive saddles $p_i,p_{i+1}$ ($p_{m+1}:=p_{1}$). Then  $\overline{W^s_{p_i}}$ and $\overline{W^s_{p_{i+1}}}$ form $\pa \mathcal{D}_i$. Recalling that in the case $N=2$ there are no heteroclinic intersections, and that $E$ is connected by Proposition~\ref{prop:efc}, we see that the curve $\mathcal{D}_i \cap E$ connecting $p_i$ and $p_{i+1}$ must contain at least one sink. 
If there were at least two such sinks, then there have to be a saddle $p_j \in \mathcal{C}_E^*$ between them. 
But then $\overline{W^s_{p_j}}$ would split $\mathcal{D}_i$ into two connected components, which contradicts our definition of $\mathcal{D}_i$. 
Thus, $\omega_i \in \mathcal{D}_i \cap E \cap \mathcal{C}_0$ is unique for each $i$. 
So, the two unstable separatrices of ${p_{i+1}}$ lie in the disks $\mathcal{D}_i$ and $\mathcal{D}_{i+1}$, respectively. 
Thus, 
$$
E=\omega_1\cup W^u_{p_2}\cup\dots\cup\omega_{m}\cup W^u_{p_1}\cup\omega_1
$$ 
is a simple closed curve, and hence $E$ is homeomorphic to $\mathbb S^1$. By Proposition~\ref{prop:efc}, $E$  divides $\Omega$ into two connected components $\Omega_-$, $\Omega_+$ such that ${\Omega_\pm}$ is bounded by the two closed curves $S_\pm$ and $E$. 
Since ${\Omega_\pm}\subset\mathbb S^2$, it is homeomorphic to $\mathbb S^1\times (0,1)$.  
\end{proof}

Unlike the case $N=2$, in higher dimensions the situation is more difficult. 
\begin{proposition}
Let $N=3$. 
Then there exists a pair $(\Omega,u)$ such that $u$ has precisely two sources, two sinks, and two saddles, and all critical points other than the sources belong to $E$, but $E$ is not a two-dimensional manifold. 
In particular, $E$ is not homeomorphic to $\mathbb S^{2}$.
\end{proposition}
\begin{proof} 
In the case $N=3$, the intersection of $E$ with the basin of a sink belonging to $E$ need not be a (topological) disk, in general, and hence $E$ does not have to be a manifold. 
This phenomenon occurs in the presence of heteroclinic intersections in the flow $\phi^t$. 
In Figure~\ref{notS}, we provide a schematic example of a phase portrait of $\phi^t$ for which the effectless cut is not a manifold because of heteroclinic intersections between the saddles $p_1$ and $p_2$. 
(The shape of such $E$ resembles the so-called \textit{cross-cap} surface, see Figure~\ref{notS2}.) 

\begin{figure}[!ht]
    \centering
    \includegraphics[width=0.85\linewidth]{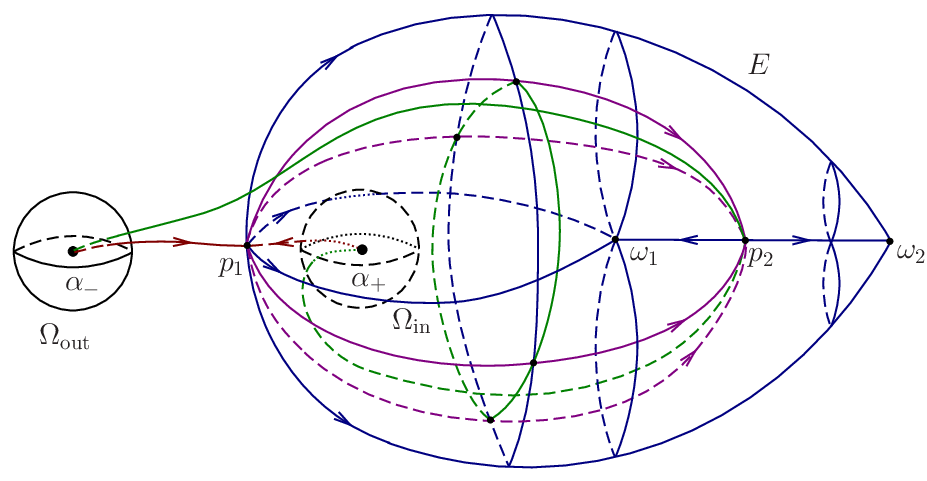}
\captionsetup{width=0.75\textwidth}
    \caption{Effectless cut (blue) $E = W_{p_1}^u \cup W_{p_2}^u$ which is not a manifold. Here, $p_1, p_2$ are saddles, and $\omega_1$, $\omega_2$ are sinks. Two green trajectories connect $p_2$ with sources $\alpha_\pm$. Green circle is a section of $W_{p_2}^s$. Four purple trajectories are heteroclinic intersections.}
    \label{notS}
\end{figure}

\begin{figure}[!ht]
    \centering
    \includegraphics[width=0.6\linewidth]{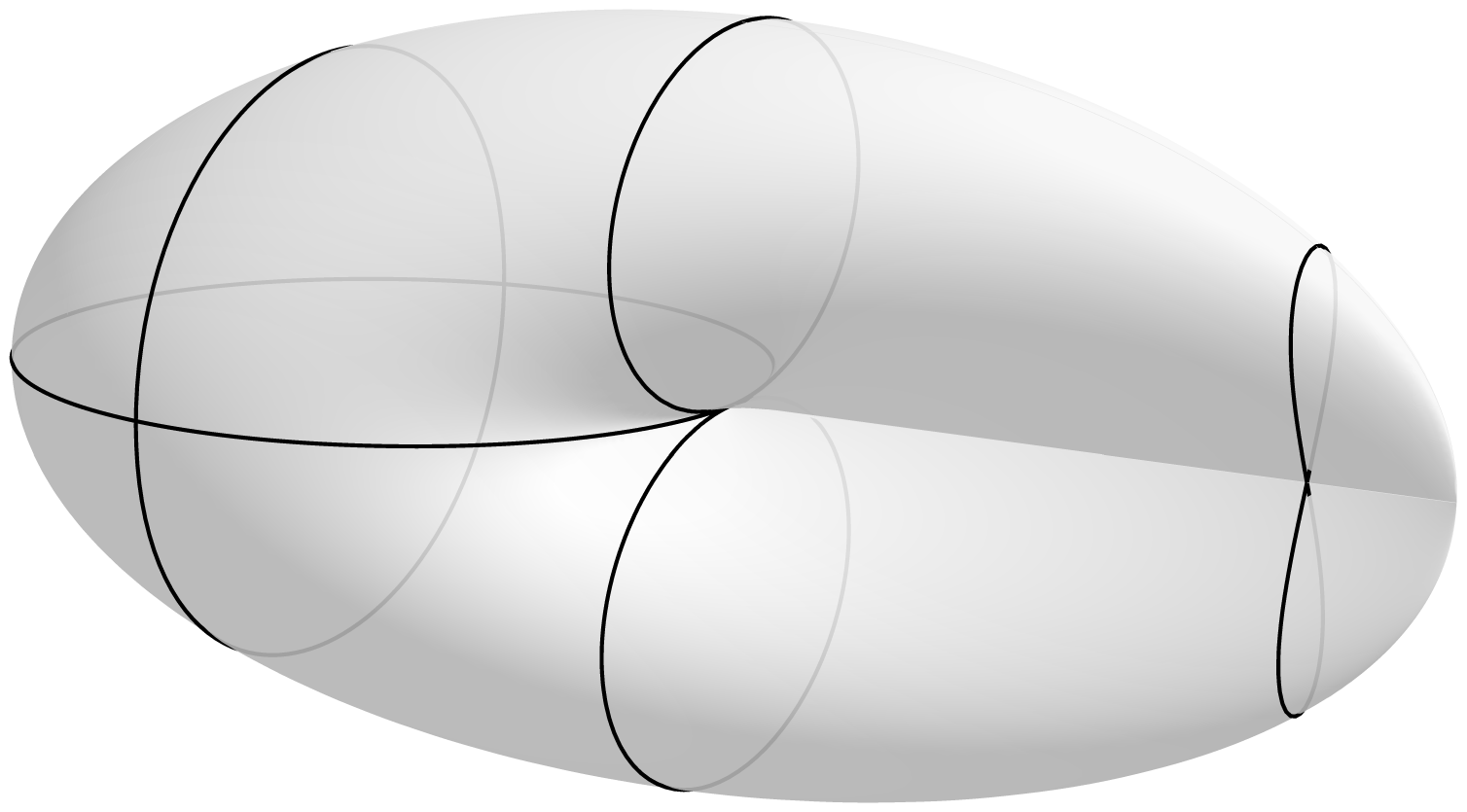}
    \caption{Cross-cap }
    \label{notS2}
\end{figure}

To complete the proof, we construct a Morse--Smale function $u$ on $\Omega$ with the stated properties. 

Let $\sigma:\mathbb{R}^+\to \left[0, 1\right]$ be a $C^\infty$-function such that $\sigma(t)=0$ for $t \leq 2$, $\sigma(t)=1$ for $t \geq 3$, and $\sigma$ is strictly increasing in $(2,3)$. 
For $x=(x_1,x_2,x_3)\in\mathbb R^3$, we denote
$$
\psi(x_3)=x_3^4 - 3 x_3^2,
\quad
z(x)=|x|^2+\psi(x_3)
\equiv
x_1^2+x_2^2+x_3^4 - 2 x_3^2,
$$ 
and let $v:\mathbb R^3\to\mathbb R$ be given by
\begin{equation}\label{eq:v}
v(x) =\left(1-\sigma(|x|^2)\right)\cdot z(x)+ \sigma(|x|^2)\cdot |x|^2. 
\end{equation} 

Let us show, by analyzing $\nabla v$, that $v$ is a Morse function with exactly three critical points: $(0,0,-1)$ (index $0$, minimum point), $(0,0,0)$ (index $1$, saddle), and $(0,0,1)$ (index $0$, minimum point).

If $|x|^2\leq 2$, then $\nabla v(x)=\nabla z(x)=\{2x_1,2x_2,4x_3(x_3^2-1)\}$, so that $v$ has exactly three critical points, all lying on the $x_3$-axis: $(0,0,-1)$, $(0,0,0)$, and $(0,0,1)$. 
To find the indices of these critical points, we compute the Hessian matrix: 
$$
\mathrm{Hess}(v(x))
= 
\begin{pmatrix} 
2 & 0 & 0 \\ 0 & 2 & 0 \\ 0 & 0 & 12x_3^2 - 4 
\end{pmatrix}.
$$ 
At $(0,0,\pm 1)$ the eigenvalues are $(2, 2, 8)$, all of which are positive; hence, the index is $0$. 
At $(0,0,0)$ the eigenvalues are $(2, 2, -4)$, there is exactly one negative eigenvalue; hence, the index is $1$. 

If $|x|^2\geq 3$, then $\nabla v(x)=\nabla (|x|^2)=\{2x_1,2x_2,2x_3\}$, so that $v$ has no critical points. 

Let $2< |x|^2 <3$. 
Then $v(x)=(1-\sigma(|x|^2))\cdot\psi(x_3)+|x|^2$. 
Denote
$$
w(x)
=
1-\psi(x_3) \cdot \left.\frac{d \sigma(t)}{d t}\right|_{t=|x|^2},
$$
so that $\frac{\partial v(x)}{\partial x_i}=2x_i\cdot w(x)$ for $i=1,2$, and $\frac{\partial v(x)}{\partial x_3}=2x_3\cdot w(x)+(1-\sigma(|x|^2))\cdot (4x_3^3-6x_3)$. 
Let a point $x_0=(x_1^0,x_2^0,x_3^0)$ be such that $|x_0|^2 \in (2,3)$ and 
$w(x_0)=0$. 
As $1-\sigma(|x_0|^2)>0$, we have  $\frac{\partial v(x)}{\partial x_3}|_{x_0}=0$ if and only if either $x^0_3=0$ or $x^0_3= \pm \sqrt{3/2}$. 
On the other hand, since $\frac{d \sigma(t)}{d t}|_{t=|x_0|^2}>0$ and  $\psi(0)=0$, $\psi(\pm \sqrt{3/2})=-9/4$, we get $w(x_0)>0$, which is impossible by our choice of $x_0$.  
Thus, every critical point of $v$, if any exists, must lie on the $x_3$-axis and its third coordinate satisfies the equation 
$$
1-(x_3^4 - 3 x_3^2) \cdot \left.\frac{d \sigma(t)}{d t}\right|_{t=x_3^2}+\left(1-\sigma(x_3^2)\right) \cdot (2x_3^2-3)=0.
$$ 
However, each term on the left-hand side of this equation is strictly positive for $2<x_3^2<3$. 
Hence, $v$ has no critical points in this region. 
This completes the verification of the number and types of critical points of $v$. 

Thus, $r=(0,0,0)$ is the saddle point of the gradient-descent flow of $v$ and its stable manifold $W^s_r$ is two-dimensional. 
By the construction, the sublevel set $B_+=\{x\in\mathbb R^3:v(x)\leq 16\}$ is the $3$-dimensional ball of radius 4 centered at the origin, and $c_r=\partial B_+\cap W^s_r$ is a circle. 
Now consider the function $32-v(x)$ on a copy $B_-$ of $B_+$. 
Let $p\in B_-$ be the saddle point of the function $32-v(x)$ and  $c_p=\partial B_-\cap W^u_p$. 
Let $g:\partial B_+\to\partial B_-$ be a diffeomorphism such that the intersection $g(c_r)\cap c_p$ consists of exactly four transverse intersection points. 
Then, for the 3-dimensional sphere $\mathbb{S}^3=B_+\cup_g B_-$, the natural projection $h:B_+\sqcup B_-\to \mathbb{S}^3$ is well-defined and the required function $u: \mathbb{S}^3 \mapsto \mathbb{R}$ is given by the formula
$$
u(h(x))=\begin{cases}
	v(x), & x\in B_+,\\
	32-v(x), & x\in B_-.
\end{cases}
$$
The proof is complete.
\end{proof}

Finally, it remains to prove the following result on the approximation of the first eigenfunction of \eqref{eq:D} by Morse--Smale first eigenfunctions of \eqref{cutproblem2d-pert}.
\begin{proposition}\label{prop:approx:MS}
Under the assumptions and with the notation of Proposition~\ref{prop:approx}, $u_n$ can additionally be chosen to be a Morse--Smale function in $\Omega$ for any $n \in \mathbb{N}$. 
\end{proposition}
\begin{proof}
The proof follows the same lines as the proof of Proposition~\ref{prop:approx}, but with the following modification. 
Instead of approximating $u$ by Morse functions $v_n$ in $\overline{\Omega}$, we may approximate $u$ by Morse--Smale functions $w_n$ in $\overline{\Omega}$, see, e.g., \cite[Remark~6.7]{banyaga2004lectures}. 
Since $u$ has no critical points in a neighborhood $U$ of $\partial \Omega$, neither does $w_n$ for any sufficiently large $n$. 
Moreover, as discussed at the beginning of this section, each $w_n$ can be considered as a function from $\mathbb{S}^N$ to $\mathbb{R}$ such that $w_n$ has exactly one maximum point $\alpha_+^n$ in $\Omi$ and one maximum point $\alpha_-^n$ in $\mathbb{S}^N \setminus \overline{\Omo}$, so that $U$ belongs to $W^u_{\alpha_+^n} \cup W^u_{\alpha_-^n}$.   
Thus, by \eqref{eq:justref}, the unstable manifold of any saddle point of $w_n$ does not intersect $U$. 
This implies that all heteroclinic intersections of $w_n$ belong to a compact subset $K$ of $\Omega$ independent of $n$, and they are transversal since $w_n$ is a Morse--Smale function. 
Thus, taking a cut-off function $\varphi$ as given in \eqref{eq:phi-properties} such that $\text{supp}(\varphi) \subset \Omega \setminus U$, we conclude that the function $u_n
= \varphi  w_n+
(1-\varphi) u $ coincides with $w_n$ in $K$ and has no critical points in $\Omega \setminus K$, so that $u_{n}$
is a Morse--Smale function. 
\end{proof}

\section{Final remarks}\label{sec:final_remarks}

\begin{remark}\label{rem:perimeter_constraint}
    As mentioned in Section \ref{sec:intro}, one can observe that in the planar , the $(N-1)$-th quermassintegral constraint in \eqref{eq:reverseFK-assumption-DN} coincides with the perimeter constraint on $\pa\Omi$. 
	However, in higher dimensions, the former is stronger than the latter. 
	Indeed, the $(N-1)$-th quermassintegral constraint implies (see, for instance, \cite[Proposition~3.4]{AnoopMrityunjoy}) that \begin{equation}\label{eq:perimeter_cons}
    |\pa \Omi|\leq |\pa B_\alpha|.
    \end{equation}
    Moreover, this inequality is strict if $\Omi$ is not a ball. 
    Now, there exist annular domains $\mathcal{A}_1:=B_\be\setminus \overline{B_\al}$ satisfying \eqref{eq:reverseFK-assumption-DN0}, and $\mathcal{A}_2:=B_{\be'}\setminus \overline{B_{\al'}}$ satisfying \eqref{eq:reverseFK-assumption-DN}. Furthermore, we must have $\be'>\be$ and $\al'>\al$ in view of \eqref{eq:perimeter_cons} (see, e.g., \cite[Remark~1.5]{AnoopMrityunjoy}) and the volume constraint appearing in \eqref{eq:reverseFK-assumption-DN}. Using the variational characterization \eqref{eq:lambda1}, one can verify that 
    \begin{equation}\label{eq:A1A2}
        \lambda_1^{\mathcal{DN}}(\mathcal{A}_1)<\lambda_1^{\mathcal{DN}}(\mathcal{A}_2).
    \end{equation}
    However, it is not known whether one can establish the following reverse Faber--Krahn inequality (for the Dirichlet-Neumann case) with the perimeter constraint (see \eqref{eq:reverseFK-assumption-DN0}):
\begin{equation}\label{open:perimeter_cons}
        \lambda_1^{\mathcal{DN}}(\Om)<\lambda_1^{\mathcal{DN}}(\mathcal{A}_1),
    \end{equation}
    which would be \textit{stronger} than \eqref{eq:reverseND} (see also, \cite[Theorem 1.6]{AnoopMrityunjoy} and \cite[Theorem 1.1]{DellaPiscitelli}) in view of \eqref{eq:A1A2}.
    We refer to \cite[Section 6, Open problem (4)]{AnoopMrityunjoy}, where the authors provide numerical evidence supporting the validity of \eqref{open:perimeter_cons} for certain domains in $\R^3$.
    Consequently, the extension of Theorem~\ref{thm:FK} to the perimeter constraint on $\pa\Omi$ instead of the $(N-1)$-th quermassintegral constraint remains unknown.
\end{remark}

\begin{remark}
	From the proof of Theorem \ref{thm:FK} (see also the proof of \cite[Theorem 1.2]{ABG2}), we observe that the inequality \eqref{eq:perimeter_cons} is insufficient for the application of the inner parallel method if the Robin parameter $\hi$ is negative. 
	Consequently, the extension of Theorem \ref{thm:FK} to the case $\hi,\ho<0$ remains an open problem in higher dimensions. In the planar case, such a result was recently established in \cite[Theorem~1.2]{ABG2}. 
\end{remark}

\begin{remark}
    In Proposition~\ref{prop:approx:MS}, we showed that the first eigenfunction $u$ of \eqref{eq:D} can be approximated by first eigenfunctions of the perturbed problem \eqref{cutproblem2d-pert} in such a way that each approximating eigenfunction is a Morse--Smale function. 
    It is natural to ask whether \textit{any} eigenfunction of \eqref{eq:D} is Morse--Smale in the case of general position, either with respect to perturbation by a potential or with respect to domain perturbation, cf.\ \cite{albert1978generic,uhl}. 
\end{remark}

	\bigskip
	\noindent
	\textbf{Acknowledgments.}
	T.V.~Anoop acknowledges the Core Research Grant  (CRG/2023/005344)  by ANRF.
    The work of O.~Pochinka was supported by the Russian Science Foundation (project 23-71-30008).
    M.~Ghosh is supported by TIFR Centre for Applicable Mathematics (TIFR-CAM). 
    V.~Bobkov and O.~Pochinka are grateful to M.~Barinova (HSE) for inspiring and motivating discussions during early stages of the project.

        \smallskip
	\noindent
	\textbf{Statements and Declarations.}
 The authors declare no potential conflict of interest.

	\smallskip
	\noindent
	\textbf{Data Availability.}
    No data was used or generated in the research described in the article.

\addcontentsline{toc}{section}{References}
\bibliographystyle{abbrvurl}
\bibliography{Reference}

\end{document}